\documentclass[a4paper,leqno,11pt,twoside]{amsart}
\usepackage{amsmath,amsfonts,amssymb,amsthm,amscd}
\usepackage[utf8]{inputenc}
\usepackage[english]{babel}
\usepackage[colorlinks=true,citecolor=blue, urlcolor=blue, linkcolor=blue,pagebackref]{hyperref}
\usepackage[top=1.2in,bottom=1.2in,left=1.in,right=1.in]{geometry} 
\usepackage{graphicx}
\usepackage{paralist}
\usepackage{tabto}
\usepackage{standalone}
\usepackage{tikz}
\usetikzlibrary{matrix}
\usetikzlibrary{arrows}
\usepackage{xfrac}

\usepackage[all]{xy} 

\usepackage{enumerate}
\usepackage{xcolor}
\usepackage{aliascnt}

\theoremstyle{plain}
\newtheorem{thm}{Theorem}[section]

\newtheorem{thmIntr}{Theorem}

\newaliascnt{corIntr}{thmIntr}
\newtheorem{corIntr}[corIntr]{Corollary}

\newaliascnt{lem}{thm}
\newtheorem{lem}[lem]{Lemma}
\aliascntresetthe{lem}

\newaliascnt{cor}{thm}
\newtheorem{cor}[cor]{Corollary}
\aliascntresetthe{cor}

\newaliascnt{prop}{thm}
\newtheorem{prop}[prop]{Proposition}
\aliascntresetthe{prop}

\newaliascnt{quest}{thm}
\newtheorem{quest}[quest]{Question}
\aliascntresetthe{quest}

\newaliascnt{conj}{thm}

\aliascntresetthe{conj}

\theoremstyle{definition}

\newaliascnt{rem}{thm}
\newtheorem{rem}[rem]{Remark}
\aliascntresetthe{rem}

\newaliascnt{defn}{thm}

\aliascntresetthe{defn}

\newaliascnt{ex}{thm}

\aliascntresetthe{ex}

\numberwithin{equation}{section}

\def\bP{\ensuremath{\mathbb{P}}}
\def\bQ{\ensuremath{\mathbb{Q}}}

\def\bZ{\ensuremath{\mathbb{Z}}}
\def\bC{\ensuremath{\mathbb{C}}}

\def\cA{\ensuremath{\mathcal{A}}}

\def\cC{\ensuremath{\mathcal{C}}}

\def\cG{\ensuremath{\mathcal{G}}}
\def\cH{\ensuremath{\mathcal{H}}}
\def\cM{\ensuremath{\mathcal{M}}}
\def\cN{\ensuremath{\mathcal{N}}}
\def\cO{\ensuremath{\mathcal{O}}}
\def\cP{\ensuremath{\mathcal{P}}}
\def\cQ{\ensuremath{\mathcal{Q}}}
\def\cR{\ensuremath{\mathcal{R}}}
\def\cT{\ensuremath{\mathcal{T}}}
\def\cU{\ensuremath{\mathcal{U}}}

\newcommand{\cV}{\mathcal V}

\newcommand{\tC}{{\widetilde C}}
\def\tcR{\ensuremath{\widetilde{\mathcal{R}}}}
\def\tcP{\ensuremath{\widetilde{\mathcal{P}}}}

\def\ocM{\ensuremath{\overline{\mathcal{M}}}}
\def\ocR{\ensuremath{\overline{\mathcal{R}}}}
\def\ocP{\ensuremath{\overline{\mathcal{P}}}}
\def\ocA{\ensuremath{\overline{\mathcal{A}}}}

\def\RGtri{\ensuremath{\cR\cG^{\phantom{1,}1}_{4,3}}}
\def\Gtetra{\ensuremath{\cG^{\phantom{1,}1}_{3,4}}}
\def\Mtetra{\ensuremath{\cM^{\phantom{1,}1}_{7,4}}}
\def\GPetri{\ensuremath{\cG\cP^{\phantom{1,}1}_{6,4}}}
\def\Gtetrop{\ensuremath{{(\cG^{\phantom{1,}1}_{3,4})}^{ns}}}

\newcommand{\lra}{\longrightarrow}

\DeclareMathOperator{\Nm}{Nm}
\DeclareMathOperator{\Pic}{Pic}

\DeclareMathOperator{\Sing}{Sing}
\DeclareMathOperator{\Aut}{Aut}
\DeclareMathOperator{\Cliff}{Cliff}
\DeclareMathOperator{\NS}{NS}

\definecolor{applegreen}{rgb}{0.55, 0.71, 0.0}

\makeatletter
\newcommand{\mylabel}[2]{#2\def\@currentlabel{#2}\label{#1}}
\makeatother

\newcommand{\set}[1]{\left\{#1\right\}}
\newcommand\restr[2]{{
 \left.\kern-\nulldelimiterspace 
 #1 
 \vphantom{\big|} 
 \right|_{#2} 
 }}

\title[Geometry of Prym semicanonical pencils]{Geometry of Prym semicanonical pencils and an application to cubic threefolds}
\author[M. Lahoz, J.C. Naranjo, and A. Rojas]{Martí Lahoz, Juan Carlos Naranjo, and Andr\'es Rojas}

\address{Departament de Matem\`atiques i Inform\`atica,
Universitat de Barcelona, Gran Via de les Corts Catalanes, 585, 08007 Barcelona, Spain}
\email{marti.lahoz@ub.edu}
\urladdr{\url{http://www.ub.edu/geomap/lahoz/}}
\email{jcnaranjo@ub.edu}
\urladdr{\url{http://webgrec.ub.edu/webpages/000006/ang/jcnaranjo.ub.edu.html}}
\curraddr{Mathematisches Institut, Universität Bonn, Endenicher
Allee 60, 53115 Bonn, Germany}
\email{arojas@math.uni-bonn.de} 
\begin{document}

\begin{abstract}
In the moduli space $\cR_g$ of double \'etale covers of curves of a fixed genus $g$, the locus formed by covers of curves with a semicanonical pencil consists of two irreducible divisors $\cT^e_g$ and $\cT^o_g$.
We study the Prym map on these divisors, which shows significant differences between them and has a rich geometry in the cases of low genus. In particular, the analysis of $\cT^o_5$ has enumerative consequences for lines on cubic threefolds.
\end{abstract}


\thanks{M. L. was supported by a Ram\'on y Cajal fellowship and A. R. was partially supported by the Spanish MINECO grants MDM-2014-0445 and RYC-2015-19175. The three authors were partially supported by the Spanish MINECO research project PID2019-104047GB-I00.}

\maketitle

\setcounter{tocdepth}{1}

\section{Introduction}

Let $\cT_g \subset \mathcal M_g$ be the subset of (isomorphism classes of) complex, smooth, irreducible curves $C$ of genus $g\ge 3$ with a semicanonical pencil, that is, with a theta-characteristic $L\in\Pic^{g-1}(C)$ such that $h^0(C,L)$ is even and positive.
This divisor was studied in \cite{te}, where Teixidor proved its irreducibility and computed the class of its closure in the rational Picard group of the Deligne-Mumford compactification $\overline {\mathcal M}_g$. 

When pulled back to the moduli space of smooth Prym curves 
\[
\cR_g=\{(C,\eta) \mid [C]\in \mathcal M_g, \eta \in JC_2\setminus\{\cO_C\}\}/\cong
\]
via the forgetful map, it breaks up into two divisors according to the parity (as a theta-characteristic) of the semicanonical pencil tensored with the $2$-torsion line bundle defining the Prym curve.
We denote by $\cT^e_g$ and $\cT^o_g$ the divisors of even and odd Prym semicanonical pencils respectively (for simplicity, we use the same notation for the divisors in $\cR_g$ and for their closures in the Deligne-Mumford compactification $\ocR_g$).

Recently, Maestro and the third author have proved in \cite{maroj} that $\cT^e_g$ and $\cT^o_g$ are irreducible for $g\neq 4$, and have computed their classes in the rational Picard group of $\ocR_g$.
The irreducibility for $g=4$ will be obtained as a consequence of our results.

The main goal of this paper is to investigate the geometry of $\cT^e_g$ and $\cT^o_g$, specially in relation to the Prym map $\cP_g:\cR_g \lra \cA_{g-1}$ and its Beauville extension $\tcP_g$ to a proper map. 
The Prym map has well-known generic fibers for $g\le 6$ and is generically injective for $g\ge 7$; nonetheless, its restriction to divisors is often far from being understood.
For instance, the restriction of $\widetilde{\cP}_g$ to the divisor of Beauville admissible covers of nodal curves has recently received attention, since its study is equivalent to that of the so-called ramified Prym map (see \cite{mp} and \cite{no}).

The problem we are facing can be found in the literature for genus $5$ in two very remarkable works.
On the one hand, the even case $\mathcal T^e_5$ was considered by Beauville in \cite{be_invent}, where he proved that $\cT^e_5$ is irreducible and equals the preimage of $\theta_{null}\subset\cA_4$ via $\tcP_5$.
On the other hand, Izadi proved that $\cT^o_5$ dominates $\cA_4$ (see \cite[Proof of Theorem~6.14]{iz}, where $\cT^o_5$ is denoted by $\theta_{null \, 2}$).
We extend this analysis to other values of $g$, and in the case of $\cT^o_5$ we provide a more detailed description with unexpected connections to the geometry of cubic threefolds.

Let us first consider the even cases $\cT^e_g$.
According to Mumford's description \cite{mu} of the singularities of the theta divisor of a Prym variety and \cite{YanoAnn}, it is well known that $\cP_g$ maps $\cT^e_g$ to the divisor $\theta_{null}\subset\cA_{g-1}$ of principally polarized abelian varieties whose theta divisor contains a singular $2$-torsion point of even multiplicity.
Combining this with results of Teixidor on the loci of curves with unexpected theta-characteristics (\cite{te2}), one proves item \eqref{mainA:preim} and part of \eqref{mainA:gg6} in the following theorem, whereas \eqref{mainA:fiber} is essentially a consequence of Recillas' trigonal construction:

\begin{thmIntr}\label{mainA}
The divisors $\cT^e_g$ of even semicanonical pencils satisfy:
\begin{enumerate}[{\rm (1)}]
\item\label{mainA:preim} $\cT^e_g=\cP_g^{-1}(\theta_{null})$ for $g\geq 3$.
\item\label{mainA:fiber} The fiber of $\cP_4$ on a general hyperelliptic Jacobian $JX\in \cA_3$ is birationally equivalent to its Kummer variety.
In particular, the divisor $\cT^e_4$ is irreducible.
\item\label{mainA:gg6}
For $g\geq6$, the restricted Prym map $\cP_g\,|_{\cT_g^e}$ is generically finite onto its image.
In particular, $\deg(\cP_6\,|_{\cT_6^e})=27$.
\end{enumerate}
\end{thmIntr}

On the other hand, the behavior of the divisors $\cT^o_g$ of odd semicanonical pencils differs considerably from that of the even cases.
Indeed, for low values of $g$ (as long as $\dim \cT^o_g\geq\dim\cA_{g-1}$), $\cT^o_g$ dominates $\cA_{g-1}$. 

Our results for the divisors $\cT^o_g$ are summarized in the following theorem.
For the case $g=5$, let us recall that Donagi \cite{do_fibres} established a birational map between $\cA_4$ and the set $\cR\cC^+$ of pairs $(V,\delta)$, where $V\subset\bP^4$ is a smooth cubic threefold and $\delta\in JV_2$ is a $2$-torsion point of its intermediate Jacobian with a certain parity condition. We will denote by $F(V)$ the Fano surface of lines contained in $V$.

\begin{thmIntr}\label{mainB}
The divisors $\cT^o_g$ of odd semicanonical pencils satisfy:
\begin{enumerate}[{\rm (1)}]
\item\label{mainB:g3} The map $\cP_3\, \vert_{\cT^o_3}:\cT^o_3 \lra \cA_2$ is dominant, and its general fiber is isomorphic to the complement in the projective plane of six lines and a smooth conic.
In particular, $\cT^o_3$ is rationally connected.
\item\label{mainB:g4} The map $\cP_4\, \vert_{\cT^o_4}:\cT^o_4 \lra \cA_3$ is surjective, and the fiber of a general Jacobian $JX$ is the complement in the projective plane of the union of the canonical model of $X$ and its $28$ bitangent lines.
Moreover, $\cT_4^o$ is irreducible and rationally connected. 
\item\label{mainB:g5} {\rm (Izadi)} The restricted Prym map $\tcP_5\,|_{\cT_5^o}$ is dominant, and the fiber at a general $(V,\delta)\in\cR\cC^+$ is a partial desingularization of the curve
$\Gamma\subset F(V)$ defined by
\[
\Gamma=\set{l\in F(V)\mid \text{There exist a $2$-plane $\pi$ and a line $r\in F(V)$ with $V\cdot\pi=l+2r$}}.
\]
\item\label{mainB:gg6} For every $g\geq6$ the restricted Prym map $\cP_g\,|_{\cT_g^o}$ is generically finite onto its image.
\end{enumerate}
\end{thmIntr}

Observe that after \eqref{mainB:gg6} two natural questions arise: the computation of the degree of $\cP_g\,|_{\cT_g^o}$ (for $g\ge 6$ but specially for $g=6$) and an intrinsic description of the divisor $\cP_6(\cT_6^o)$ in $\cA_5$.
With respect to the first question, we are only able to prove that the degree of $\cP_6\,|_{\cT_6^o}$ is strictly smaller than $27$.
Moreover, in Section~\ref{sec:genus6imes} we propose a possible natural geometric description for $\cP_6(\cT_6^o)$.

The main tool for our reproof of \eqref{mainB:g5} is the class of $\cT^o_5$ in $\Pic(\ocR_5)_\bQ$.
Simultaneously, this approach reveals interesting properties of a general cubic threefold $V\subset\bP^4$, which a priori seem difficult to detect via more direct techniques.
Indeed, Donagi's description of the general fiber of $\tcP_5$ (\cite{do_fibres}) realizes (a double cover of) the Fano surface of lines $F(V)$ as a subvariety of $\ocR_5$, where the rational Picard group and the canonical class are well understood.
This enables us to prove:

\begin{thmIntr}\label{mainC}
For every smooth cubic threefold $V\subset\bP^4$, the curve $\Gamma\subset F(V)$ is numerically equivalent to $8K_{F(V)}$.
Furthermore, for $V$ general, $\Gamma$ is irreducible and its singular locus consists of $1485$ nodes.
\end{thmIntr}

To the best of the authors' knowledge, the numerical class of $\Gamma$ has never been computed. 
On the other hand, its natural counterpart
\[
\Gamma'=\set{r\in F(V)\mid \text{There exist a $2$-plane $\pi$ and a line $l\in F(V)$ with $V\cdot\pi=l+2r$}}
\]
(namely the curve formed by \emph{lines of the second type}) had been greatly studied in the literature (see e.g. \cite[Section~10]{cg}).
\autoref{mainC} has immediate consequences for the enumerative geometry of lines on a cubic threefold.
For instance, the geometric interpretation of the nodes of $\Gamma$ and the intersection points of $\Gamma$ with $\Gamma'$ establishes the following result:

\begin{corIntr}\label{mainD}
For a general smooth cubic threefold $V\subset\bP^4$, the following statements hold:
\begin{enumerate}[{\rm (1)}]
\item\label{mainD:nodes} There are exactly $1485$ lines $l\subset V$ for which there exist $2$-planes $\pi_1,\pi_2\subset\bP^4$ and lines $r_1,r_2\subset V$ satisfying $V\cdot\pi_i=l+2r_i$ $(i=1,2)$.
\item\label{mainD:inters} There are exactly 720 lines $l\subset V$ for which there exist $2$-planes $\pi_1,\pi_2\subset\bP^4$ and lines $r_1,r_2\subset V$ satisfying $V\cdot\pi_1=l+2r_1$ and $V\cdot\pi_2=2l+r_2$.
\end{enumerate}
\end{corIntr}

\textbf{Structure of the paper.} After some preliminaries on the divisors $\cT^e_g$ and $\cT^o_g$ and on Prym varieties, in Section~\ref{sec:Te} we determine $\cP_g^{-1}\big(\theta_{null}\big)$, which proves \autoref{mainA}.\eqref{mainA:preim}.
The rest of the paper essentially deals with the odd cases, and the study of each genus occupies a section.

Section~\ref{sec:genus3} is devoted to \autoref{mainB}.\eqref{mainB:g3}, whose proof is based on Mumford's results on Prym varieties of covers of hyperelliptic curves.
In Section~\ref{sec:genus4} we study the case of genus 4. Using involutions on certain moduli spaces, we prove that $\cT^o_4$ corresponds under Recillas' trigonal construction to smooth genus 3 curves endowed with a non-complete $g^1_4$ linear series, which gives the arguments for proving \autoref{mainB}.\eqref{mainB:g4}.
In this section, we also prove \autoref{mainA}.\eqref{mainA:fiber}.

Section~\ref{sec:genus5} addresses the case of genus $5$.
As explained above, we prove \autoref{mainB}.\eqref{mainB:g5} using a cohomological approach.
A more detailed analysis also gives \autoref{mainC} and several enumerative consequences, including \autoref{mainD} as well as a more precise description of the desingularization appearing in \autoref{mainB}.\eqref{mainB:g5} (see \autoref{DesingPrecisa}). 
Finally, in Section~\ref{sec:genus6imes} we study the cases of genus $g\geq6$.
After proving \autoref{mainB}.\eqref{mainB:gg6} and \autoref{mainA}.\eqref{mainA:gg6}, we propose a natural geometric description for the divisor $\cP_6(\cT_6^o)\subset \cA_5$, based on a close relation between $\cT_g^o$ and the locus of Prym curves $(C,\eta)$ for which the Brill-Noether-Prym locus $V^2(C,\eta)$ is singular.

\vspace{2mm}

\textbf{Acknowledgements.} We are pleased to thank Alessandro Verra for his useful comments at the beginning of this project.
We are also grateful to Daniel Huybrechts for pointing out an inaccuracy in a previous version, and to the referees for their careful reading of the paper.

\section{Preliminaries}\label{sec:prelim}

We work over the complex numbers.
As usual, a point of a variety is called \textit{general} (resp. \textit{very general}) if it lies outside a nontrivial Zariski-closed subset (resp. outside a union of countably many nontrivial Zariski-closed subsets).
Given an abelian variety $A$, we denote by $A_2$ the subgroup of $2$-torsion points of $A$. 

\subsection{The divisors of Prym semicanonical pencils}

If $C$ is a smooth curve of genus $g\geq3$, by \emph{semicanonical pencil} on $C$ we mean an even, effective theta-characteristic.
The locus of curves with a semicanonical pencil defines an irreducible divisor $\cT_g\subset\cM_g$, whose general element $C$ admits a unique semicanonical pencil $L$ and satisfies $h^0(C,L)=2$ (see \cite{te2,te}).
In the literature our semicanonical pencils are called \emph{vanishing theta-nulls}, but later we will use the theta-null divisor in the context of abelian varieties, and we prefer to call semicanonical pencils the even effective theta-characteristics
to avoid any confusion.

Let $\cR_g$ denote the moduli space of smooth Prym curves of genus $g$.
Since the parity of theta-characteristics remains constant in families (\cite{mu3}), the pullback of $\cT_g$ by the natural forgetful map $\pi:\cR_g\to\cM_g$ decomposes as $\pi^{-1}(\cT_g)=\cT_{g}^e\cup\cT_{g}^o$, where
\[
\begin{aligned}
&\cT^e_g=\set{(C,\eta)\in\cR_g\mid C\text{ has a semicanonical pencil $L$ with $h^0(C,L\otimes\eta)$ even}}\\
&\cT^o_g=\set{(C,\eta)\in\cR_g\mid C\text{ has a semicanonical pencil $L$ with $h^0(C,L\otimes\eta)$ odd}}
\end{aligned}
\]

Abusing notation, $\cT^e_g$ and $\cT^o_g$ will also denote the closures of these divisors in the Deligne-Mumford compactification $\ocR_g$ or in Beauville's partial compactification $\tcR_g$ by admissible covers.

The classes of $\cT^e_g$ and $\cT^o_g$ in the rational Picard group $\Pic(\ocR_g)_\bQ$ have recently been computed in \cite{maroj}.
The reader is referred to \cite[Section~1]{fa-lu} for the definition of the classes $\lambda, \delta_0', \delta_0'', \delta_0^{ram}, \delta_i, \delta_{g-i},\delta_{i:g-i}$ ($1\leq i\leq[g/2]$) generating $\Pic(\ocR_g)_\bQ$.
We do not specify the coefficients of $\delta_i, \delta_{g-i}, \delta_{i:g-i}$ since they are not useful for us.

\begin{thm}[\cite{maroj}] \label{class}
Let $[\cT^e_g],[\cT^o_g]\in\Pic(\ocR_g)_\bQ$ denote the cohomology classes of $\cT^e_g$, $\cT^o_g$ in the Deligne-Mumford compactification $\ocR_g$. Then, the following equalities hold:
\begin{align*}
 [\cT^e_g]&=2^{g-3}(2^{g-1}+1) \lambda -2^{2g-7} \delta_0'- 2^{g-5}(2^{g-1}+1) \delta_0^{ram} -\ldots \\
 [\cT^o_g]&=
 2^{2g-4}
 \lambda - 2^{2g-7} \delta_0'- 
 2^{2g-6} \delta_0''-
 2^{g-5}(2^{g-1}-1)
 \delta_0^{ram} -\ldots
 \end{align*}
Furthermore, for $g\neq4$ the divisors $\cT^e_g$ and $\cT^o_g$ are irreducible.
\end{thm}

\subsection{Prym varieties and the Andreotti-Mayer locus}\label{prelim:Prym}

We recall some basic facts on Prym varieties, most of them coming from the seminal work \cite{mu}. 

Given a smooth Prym curve $(C,\eta)\in\cR_g$, denote by $f:\tC \lra C$ the \'etale, smooth, irreducible double cover associated to $\eta$. 
The kernel of the norm map $\Nm_f:J\tC \lra JC$ breaks into two connected components; the \emph{Prym variety} $P=P(C,\eta)$ of $(C,\eta)$ is the component containing the origin of $J\tC$. Its 2-torsion subgroup $P_2$ fits into a short exact sequence
\[
0\longrightarrow\langle\eta\rangle\longrightarrow \langle\eta\rangle^\perp\xrightarrow{f^*} P_2\longrightarrow 0
\]
where $\langle\eta\rangle^\perp\subset JC_2$ denotes the orthogonal of $\langle\eta\rangle$ with respect to the Weil pairing on $JC_2$ (see \cite[Page~332, Corollary~1]{mu}).
Furthermore, the principal polarization on $J\tC$ restricts to twice a principal polarization on $P$, giving rise to the \emph{Prym map}
\[
\cP_g:\cR_g \lra \cA_{g-1}.
\]

The singularities of a theta divisor $\Xi$ representing the principal polarization can be described as follows. 
Let 
\[
 P^+=\{M \in \Pic^{2g-2}(\tC) \mid \Nm_f (M)= \omega _C\text{ and }h^0(\tC,M) \text{ is even}\}
\]
be a ``canonical presentation'' of the Prym variety $P$ in $\Pic^{2g-2}(\tC)=\Pic^{g(\tC)-1}(\tC)$, and let $\Theta_{\tC}\subset\Pic^{2g-2}(\tC)$ denote the canonical theta divisor of $J\tC$. Then $P^+\cdot\Theta_{\tC}=2\Xi^+$, where 
\[
\Xi^+=\{M \in P^+ \mid h^0(\tC,M)\geq2\}
\]
is a canonical presentation of the theta divisor of $P$, and singularities of $\Xi^+$ may arise in two different situations:
\emph{stable singularities} given by $M\in \Xi^+$ with $h^0(\tC, M)\ge 4$, and \emph{exceptional singularities} of the form $M=f^*L\otimes A \in \Xi ^+$ with $h^0(C,L)\geq 2$ and $h^0(\tC,A)>0$.

An elementary example of exceptional singularity is given by $f^*L$, where $L$ is a semicanonical pencil on $C$ such that $h^0(\tC,f^*L)=h^0(C,L)+h^0(C,L\otimes \eta)$ is even (i.e., $h^0(C,L\otimes \eta)$ is even). 
Therefore $\cP_g(\cT^e_g)\subset \cN_0$, where $\cN_0\subset \cA_{g-1}$ denotes the \emph{Andreotti-Mayer locus} of principally polarized abelian varieties (ppav) whose theta divisor has singularities.

More precisely one has $\cP_g(\cT^e_g)\subset \theta_{null}$, where $\theta_{null}\subset\cA_{g-1}$ is the divisor of ppav whose (symmetric) theta divisor contains a singular $2$-torsion point of even multiplicity.
Indeed, the singularity being a 2-torsion point follows from the fact that the symmetric models of the theta divisor in $P\subset\Pic^0(\tC)$ are obtained when $(P^+,\Xi^+)$ is translated by a theta-characteristic lying in $P^+$;
in particular, the $2$-torsion points of $P$ in the canonical model $P^+$ are the theta-characteristics of $\tC$ lying in $P^+$.
Moreover, the multiplicity is two for the general element in $\cP_g(\cT^e_g)$ thanks to \cite[Theorem~2]{YanoAnn}, so even in any case. 

Note that $\cN_0=\theta_{null}=\cA_1 \times \cA_1$ in $\cA_2$, and $\cN_0=\theta_{null}\subset\cA_3$ is the divisor of hyperelliptic Jacobians. 
For $g\geq5$, the Andreotti-Mayer locus of $\cA_{g-1}$ is the union of two irreducible divisors (\cite{mu2,deb}): 
\[
\cN_0=\theta_{null}\cup\cN_0^\prime.
\]

Whereas the theta divisor of the general element of $\theta_{null}$ has a unique singular point (which is $2$-torsion), the theta divisor of a general element of $\cN_0^\prime$ has exactly two singular (opposite) points. 
Using this fact, Mumford computed the multiplicity of each component (see \cite{mu2}), proving the following equality as cycles: $\cN_0=\theta_{null}+2\cN_0^\prime$.

For $g=5$ this was already proved in \cite{be_invent}, where Beauville showed that $\cN_0\subset\cA_4$ has two irreducible components, namely the Jacobian locus and the divisor $\theta_{null}$.
In that paper, Beauville extended $\cP_g$ to a proper map
\[
 \tcP_g: \tcR_g \lra \cA_{g-1}
\]
by considering \emph{admissible covers} of stable curves;
$\tcR_g\subset \ocR_g$ denotes the moduli space of admissible covers.

\subsection{Brill-Noether loci on Prym varieties}\label{prelim:BN} 

In contrast to the even case, an odd semicanonical pencil $L$ for a smooth Prym curve $(C,\eta)$ does not provide singularities in the canonical theta divisor $\Xi^+$, since the pullback $f^*L$ lands in the other component of $\Nm_f^{-1}(\omega_C)$:
 \[
 P^-=\set{M \in \Pic^{2g-2}(\tC) \mid \operatorname{Nm}_f (M)= \omega _C \text{ and } h^0(\tC,M) \text{ is odd}}.
 \]
To understand the situation, following Welters \cite{welters}, we consider the \emph{Brill-Noether-Prym loci}
\[
V^r(C,\eta):=\set{M \in \Pic^{2g-2}(\tC) \mid \operatorname{Nm}_f (M)= \omega _C, \; h^0(\tC,M)\geq r+1,\;h^0(\tC,M)\equiv r+1\text{(mod 2)}}
\]
with the scheme structure defined by $P^+\cap W^r_{2g-2}(\tC)$ ($r$ odd) or $P^-\cap W^r_{2g-2}(\tC)$ ($r$ even).

For instance $V^0(C,\eta)=P ^-$, $V^1(C, \eta)=\Xi^+$ and $V^3(C,\eta)\subset \Xi^+$ is the set of stable singularities.
Moreover, for $g\geq4$ and $C$ non-hyperelliptic the scheme-theoretic equality $V^2(C,\eta)=T(\tC)$ holds (see \cite[Theorem~A]{ln}, or \cite{iz} for a first set-theoretic version when $g=5$). Here the \emph{theta-dual} $T(\tC)$ parametrizes the translates of the Abel-Prym curve $\tC\subset P$ contained in the theta divisor.

Observe that for $(C,\eta )\in \cT^o_g$ with an odd semicanonical pencil $L$, we have $f^*L\in V^2(C,\eta)$.
Moreover, $f^*L$ is a singular point of $V^2(C,\eta)$ thanks to the following result, which is essentially an application of \cite[Lemma~3.1]{ho}:

\begin{lem}\label{V2}
Let $(C,\eta)\in\cR_g$ be a non-hyperelliptic Prym curve of genus $g\geq5$.
If $M\in V^2(C,\eta)\setminus V^4(C,\eta)$, then $M$ is a singular point of $V^2(C,\eta)$ if and only if $M=f^*L\otimes A$, for line bundles $L$ and $A$ satisfying $h^0(C,L)\geq2$ and $h^0(\tC,A)>0$.
\end{lem}
\begin{proof}
Since $V^2(C,\eta)$ has pure dimension $g-4$ (see \cite[Lemma~4.1]{ln}), the ``only if'' part is exactly the statement of \cite[Lemma~3.1]{ho} for $r=2$.
Following that proof, the converse statement is obtained as well, if one uses that every element of $\bigwedge^2 H^0(\tC,M)$ is decomposable by the assumption $h^0(\tC,M)=3$.
\end{proof}

\section{Even semicanonical pencils and the theta-null divisor}\label{sec:Te}

In this short section, we specify the covers in $\cR_g$ whose Prym variety has a theta divisor with a singular 2-torsion point (of arbitrary multiplicity), and we will deduce that $\cP^{-1}_g(\theta_{null})$ equals the divisor $\cT^e_g$.

\begin{prop}\label{thetanull} For every $g\geq3$, the preimage under the Prym map of the locus of ppav whose theta divisor has a singular 2-torsion point is
\begin{equation}\label{eqn:primage}
\cT^e_g\cup \set{(C,\eta)\in\cR_g\left| \begin{array}{l} C\text{ has an odd theta-characteristic $L$ such that}\\
\text{$h^0(C,L)\geq3$ and $L\otimes\eta$ is also odd}\end{array}\right.}.
\end{equation}
\end{prop}
\begin{proof}
Let $\tC\xrightarrow{f} C$ be the double \'etale cover defined by $(C,\eta)\in\cR_g$, and let $\sigma:\tC\to\tC$ be the induced involution exchanging sheets.
Throughout this proof, we consider the canonical presentation $(P^+,\Xi^+)$ in $\Pic^{2g-2}(\tC)$ of the Prym variety; recall that the $2$-torsion points of $P$ correspond to theta-characteristics of $\tC$ lying in $P^+$.

We have already seen in Section \ref{prelim:Prym} that the inclusion $\cT^e_g\subset\cP_g^{-1}(\theta_{null})$ holds. Moreover, if $C$ has an odd theta-characteristic $L$ with $h^0(C,L)\geq3$ and $L\otimes\eta$ odd, then
\[
h^0(\tC,f^*L)=h^0(C,L)+h^0(C,L\otimes\eta)\geq4
\]
and hence $f^*L$ is a theta-characteristic on $\tC$ defining a stable singularity of $\Xi^+$.

Therefore, to finish the proof it suffices to check that if $(C,\eta)\notin\cT^e_g$ and $\cP_g(C,\eta)$ has a theta divisor with a singular 2-torsion point, then $C$ has an odd theta-characteristic $L$ with $h^0(C,L)\geq3$ and $L\otimes\eta$ odd. 

So let $M\in\Xi^+$ be a singular point, corresponding to a theta-characteristic on $\tC$.
If $h^0(\tC,M)=2$, then the singularity $M$ is exceptional with $M=f^*L\otimes A$ and $h^0(C,L)= 2$, $h^0(\tC,A)>0$. Let us check that this cannot happen under the assumption $(C,\eta)\notin\cT^e_g$.

Indeed, if $\deg L=g-1$ then $A=\cO_{\tC}$ and $M=f^*L$. Since $\omega_C=\Nm_f(M)=L^{\otimes2}$, it follows that either $L$ or $L\otimes\eta$ is an even semicanonical pencil for the cover $f$, which is a contradiction.
If $\deg L<g-1$, then, under the assumption $M^2=\omega_{\tC}$, we have
\[
f^*L^2\otimes A\otimes\sigma(A)=f^*(\Nm_f(M))=f^*\omega_C=\omega_{\tC}=M^2=f^*L^2\otimes A^2
\]
and therefore $A$ is invariant by the action of the involution $\sigma$.
This allows us to express $M=f^*(L')$ for a line bundle $L'$ of degree $g-1$, which again leads to a contradiction.

Now assume that $M\in\Xi^+$ is defining a stable singularity, namely $h^0(\tC,M)\geq4$.
There is a chain of equalitites
\[
M^2=\omega_{\tC}=f^*\omega_C=f^*\Nm_f(M)=M\otimes\sigma(M)
\]
giving $M=\sigma(M)$, which is equivalent to $M=f^*L$ for a line bundle $L$ of degree $g-1$ on $C$.
Moreover, the condition $\Nm_f(M)=\omega_C$ reads as $L$ being a theta-characteristic on $C$, for which
\[
4\leq h^0(\tC,f^*L)=h^0(C,L)+h^0(C,L\otimes\eta).
\]
By the assumption $M\in\Xi^+$, the two summands must have the same parity, and cannot be even since $(C,\eta)\not\in\cT^e_g$.
It follows that the summands must be odd, which finishes the proof.
\end{proof}

The following corollary includes \autoref{mainA}.\eqref{mainA:preim}:

\begin{cor}\label{thetanull2} The divisor of even semicanonical pencils $\cT^e_g$ satisfies:
\begin{enumerate}[{\rm (1)}]
 \item\label{thetanull2:from3} For $g\geq3$, the equality $\cT^e_g=\cP_g^{-1}\big(\theta_{null}\big)$ holds.
 \item\label{thetanull2:from6} For $g\geq 6$,
 the irreducible components of \eqref{eqn:primage} of \autoref{thetanull} outside $\cT_g^e$ have codimension $3$.
\end{enumerate}
\end{cor}
\begin{proof}
Let us first prove \eqref{thetanull2:from3} when $3\leq g\leq 5$, by checking that in this range the locus \eqref{eqn:primage} described in \autoref{thetanull} equals $\cT^e_g$. Indeed, a smooth curve $C$ of genus $g\leq4$ has no theta-characteristic $L$ with $h^0(C,L)\geq3$ (otherwise it would contradict Clifford's theorem). 
When $g=5$, such a theta-characteristic is necessarily a $g^2_4$, so $C$ must be hyperelliptic (it has Clifford index $0$).
Since covers of hyperelliptic curves are contained in $\cT^e_g$ for $g\geq4$, the claim follows.

Item \eqref{thetanull2:from6} is a direct consequence of \cite[Theorem~2.17]{te2}. Finally, to prove \eqref{thetanull2:from3} for $g\geq6$ one simply combines \eqref{thetanull2:from6} with the observation that $\cP_g^{-1}(\theta_{null})$ must have pure codimension 1 (since it is the preimage of a divisor). 
\end{proof}

\begin{rem}\label{mult3}
Let $(C,\eta)\in\cR_g$ be a general point of one of the codimension 3 components of \eqref{eqn:primage} of \autoref{thetanull}.
Again by \cite[Theorem~2.17]{te2}, there exists a unique odd theta-characteristic $L$ on $C$ with $h^0(C,L)\geq3$, which satisfies $h^0(C,L)=3$ and $h^0(C,L\otimes\eta)=1$.
Let $\tC\xrightarrow{f} C$ denote the double \'etale cover associated to $(C,\eta)$.
The $2$-torsion point $M=f^*L\in\Xi^+$ is both a stable and an exceptional singularity, having multiplicity 3 in $\Xi^+$ by \cite[Theorem~2]{YanoAnn}.

When $g=6$, there is a classical construction providing one of these codimension 3 components. For a smooth cubic threefold $V\subset\bP^4$, let $JV$ denote its intermediate Jacobian; its (canonical) theta divisor has multiplicity $3$ at the origin, which is its unique singularity (see \cite{cg} and \cite{be_sing}).

Over the 10-dimensional locus $\cC\subset\cA_5$ of intermediate Jacobians of cubic threefolds, the Prym map $\cP_6$ fails to be finite: the fiber $\cP_6^{-1}(JV)$ is $2$-dimensional, given by (an open subset of) the Fano surface of lines on $V$ (\cite[Part~V]{donsmith}).
Then the preimage $\cP_6^{-1}(\cC)$ equals
\[
\cR\cQ^-=\set{(Q,\eta)\in\cR_6\mid Q\text{ is a smooth plane quintic, }h^0(Q,\cO_Q(1)\otimes\eta)\text{ is odd}}
\]
and, since $\cR\cQ^-$ is not contained in $\cT^e_6$ (a general quintic admits no semicanonical pencil), it follows that $\cR\cQ^-$ is one of the codimension 3 components mentioned above.
\end{rem}

\section{Genus 3 and hyperelliptic Prym curves} 
\label{sec:genus3}

In the case of smooth curves of genus $3$, a semicanonical pencil is the same as a $g^1_2$, so the divisor $\cT_3\subset \mathcal M_3$ equals the hyperelliptic locus $\cH_3$.
At the level of Prym curves, the irreducible divisors $\cT^e_3$ and $\cT^o_3$ admit an easy description in terms of the number of Weierstrass points needed to express the $2$-torsion line bundle (see \cite[Example~2.1]{maroj}).

Recall that the Prym map (on smooth covers) $\cP_3:\cR_3\to\cA_2$ is surjective. When we consider its restriction to $\cT_3^e$ and $\cT_3^o$, two distinct behaviors arise.
On the one hand, $\cT^e_3=\cP_3^{-1}(\theta_{null})$ as we saw in \autoref{thetanull2}.\eqref{thetanull2:from3}, where $\theta_{null}=\cA_1\times\cA_1\subset\cA_2$ is the locus formed by products of elliptic curves.
On the other hand, as stated in \autoref{mainB}.\eqref{mainB:g3}, the restriction to $\cT^o_3$ is dominant with general fiber isomorphic to the complement in $\bP^2$ of six lines and a smooth conic.

\begin{proof}[{Proof of \autoref{mainB}.\eqref{mainB:g3}}]
Let $C'$ be a smooth curve of genus 2.
Since $C'$ is hyperelliptic, by \cite[page~346]{mu}, expressing $JC'$ as the Prym of a cover of a genus $3$ hyperelliptic curve $C$ is equivalent to the construction of a diagram
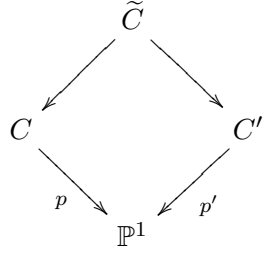
\[
\xymatrix{&\tC\ar[ld]\ar[rd]\\
C\ar[rd]_p&&C'\ar[ld]^{p'}\\
&\bP^1
}
\]
where: the double cover $p$ is branched at the six branch points of $p'$ and two extra points, and $\tC$ is the normalization of $C\times_{\bP^1}C'$.
This proves the dominance of $\cP_3\, \vert_{\cT_3^o}$.

In order to determine the fiber, we consider a curve $C'$ which is general in the following sense: 
if $p_1,\ldots,p_6$ are the branch points of the double cover $p':C'\to\bP^1$, then there is no nontrivial projective transformation of $\bP^1$ mapping four points of $\set{p_1,\ldots,p_6}$ to four points (possibly not the same four) in $\set{p_1,\ldots,p_6}$.
Under this assumption, the fiber $\cP_3\, \vert_{\cT_3^o}^{-1}(JC')$ parametrizes all the possible choices of two non-repeated points in $\bP^1\setminus\{p_1,\ldots ,p_6\}$, according to the previous description.

Consider the natural isomorphism $(\bP^1)^{(2)}\cong{(\bP^2)}^*$ identifying a pair of points on a smooth plane conic with the line joining them.
Under this identification $(\bP^1\setminus\{p_1,\ldots ,p_6\})^{(2)}$ is isomorphic to the complement in $(\bP^2)^*$ of six lines (no three of them concurrent), since we are considering lines passing through none of the six marked points of the conic.

To finish the description of the general fiber $\cP_3\, \vert_{\cT_3^o}^{-1}(JC')$, simply note that we are avoiding lines that are tangent to the conic as well, since we are considering pairs in $(\bP^1\setminus\set{p_1,\ldots ,p_6})^{(2)}$ formed by two distinct points.

Finally, since a variety which dominates a rationally connected variety with rationally connected generic fibers is rationally connected (\cite{ghs}), we obtain that $\cT_3^o$ is rationally connected.
\end{proof}

\begin{rem}We can be more precise regarding the image of $\cT^o_3$ via the Prym map and the rational connectedness of $\cT^e_3$ and $\cT^o_3$:
\begin{enumerate}[{\rm (1)}]
\item Since the semicanonical pencil on a curve of $\cT_3$ is unique, the divisors $\cT^e_3$ and $\cT^o_3$ are disjoint.
Therefore, thanks to the equality $\cP_3^{-1}(\theta_{null})=\cT^e_3$ we obtain that $\cP_3\, \vert_{\cT_3^o}$ is dominant, but not surjective.
Indeed, $\cP_3(\cT_3^o)=\cA_2\setminus\theta_{null}$.

\item It is not difficult to prove directly that both divisors $\cT^e_3$ and $\cT^o_3$ are rationally connected using their description in terms of Weierstrass points.
Indeed, any two smooth Prym curves of $\cT^e_3$ (resp.~$\cT^o_3$) can be connected by a chain of (at most five) rational curves contained in $\cT^e_3\subset\ocR_3$ (resp.~$\cT^o_3\subset\ocR_3$);
basically, each rational curve parametrizes hyperelliptic Prym curves with all but one of its branch points remaining constant.
\end{enumerate}
\end{rem}

\section{Genus 4 and Recillas' trigonal construction} \label{sec:genus4}

In this section, we carry out an analysis of the case of genus $4$, which leads to the proof of \autoref{mainB}.\eqref{mainB:g4} and \autoref{mainA}.\eqref{mainA:fiber}. 
Recillas' trigonal construction provides an isomorphism between two moduli spaces, each of them endowed with a natural involution.
Our strategy exploits the fact that these involutions are compatible and that the divisors $\cT^o_4$ and $\cT^e_4$ are contained in the locus of fixed points of one of these involutions.

Recall that a smooth hyperelliptic curve $C$ of genus $4$ has ten distinct semicanonical pencils, corresponding to the sum of the $g^1_2$ as movable part with a Weierstrass base point. 
For non-hyperelliptic $C$, the canonical map $C\to\bP^3$ embeds $C$ as the complete intersection of a quadric $Q$ and a cubic surface $S$.

If $Q$ is smooth, then the curve $C$ has exactly two $g^1_3$, which parametrize the intersection of $S$ with the lines in each of the rulings of $Q$.
Observe that the sum of the two $g^1_3$ is the canonical system of $C$ and the curve $C$ has no semicanonical pencil.
On the other hand, if the quadric $Q$ is singular, then $C$ has a unique $g^1_3$ which is a semicanonical pencil;
moreover, $g^1_3$ is given by the intersections of $S$ with the system of lines in $Q$ containing the singular point.

It follows that  $\cH_4\subset\cT_4\subset\cM_4$, where $\cH_4$ stands for the hyperelliptic locus. Notice that $\cT_4$ is the closure of the locus of non-hyperelliptic curves whose canonical model is contained in a singular quadric.
Moreover, since the semicanonical pencil of a non-hyperelliptic curve of $\cT_4$ is unique, we have $\cR\cH_4=\cT^e_4\cap\cT^o_4$, where $\cR\cH_4=\pi^{-1}(\cH_4)\subset\cR_4$ denotes the locus of hyperelliptic Prym curves.

Now we address the problem of understanding the restriction of the Prym map $\cP_4:\cR_4\to\cA_3$ to the divisors $\cT^e_4$ and $\cT^o_4$.
Consider the following moduli spaces:
\begin{align*}
 \RGtri&=\{(C,\eta, M) \mid (C,\eta)\in\cR_4\setminus\cR\cH_4 \text{ and $M$ is a $g^1_3$ on $C$}\}/\cong \\
 \Gtetra&=\{(X,L) \mid X\in\cM_3 \text{ and $L$ is a (not necessarily complete) base-point-free $g^1_4$ on $X$}\}/\cong .
\end{align*}

That is, $\RGtri$ parametrizes (isomorphism classes of) covers of non-hyperelliptic genus $4$ curves endowed with a $g^1_3$, and $\Gtetra$ parametrizes genus $3$ curves endowed with a base-point-free $g^1_4$.
Both moduli spaces have projection maps forgetting the linear series:

\begin{itemize}[\textbullet]
 \item The projection $\RGtri\xrightarrow{\varphi}\cR_4\setminus\cR\cH_4$ is generically finite of degree $2$.
 Moreover, $\RGtri$ carries a natural involution $\sigma$ defined by \begin{equation}\label{eq:defsig}\sigma(C,\eta,M)=(C,\eta,\omega_C\otimes M^{-1}),\end{equation}
 which exchanges the two sheets of the open subset of $\RGtri$ where $\varphi$ is finite.
 
 \item Let us study the fiber of the projection $\Gtetra\xrightarrow{\psi}\cM_3$ over a curve $X\in\cM_3$.
 First of all, note that the scheme $\cG^1_4(X)$ of $g^1_4$ linear series on $X$ is easily identified with the blow-up of $\Pic^4(X)$ at the canonical sheaf $\omega_X$.
 This scheme carries a natural involution, given by $L\mapsto\omega_X^2\otimes L^{-1}$;
 indeed, this involution, defined outside the exceptional divisor of $\cG^1_4(X)$, extends as the identity on the exceptional divisor as proved in \cite[Proposition~6.1]{fns}.
 
 \noindent If $X$ is non-hyperelliptic and we regard it as a quartic plane curve, the $g^1_4$'s on $X$ with base points are exactly those given by pencils of lines through points of $X$.
 Linear series with base points are thus parametrized by $X$, and are contained in the exceptional divisor of $\cG^1_4(X)$.
 
 \noindent If $X$ is hyperelliptic, tetragonal series with base points are those of the form $g^1_2+p+q$ and hence the open subset of $\cG^1_4(X)$ parametrizing series without base points is the complement of a copy of $X^{(2)}$.

 \noindent According to this description, it follows that for any smooth curve $X$ the open subset of $\cG^1_4(X)$ formed by base-point-free tetragonal series remains invariant by the involution.
 Since this involution is compatible as well with the automorphisms of $X$, we obtain an involution $\tau:\Gtetra\to\Gtetra$ defined by \begin{equation}\label{eq:deftau}\tau(X,L)=(X,\omega_X^2\otimes L^{-1}).\end{equation}
\end{itemize}

Recillas' trigonal construction (cf. \cite{rec}) yields a morphism $R:\RGtri\longrightarrow \Gtetra$ making the following diagram commutative:
\[
 \begin{xy}
\xymatrix{
\RGtri
\ar[rr]^R\ar[d]_\varphi&&\Gtetra
\ar[d]^\psi\\
\cR_4\setminus\cR\cH_4\ar[rd]_{\cP_4}&&\cM_3\ar[ld]^{\text{Torelli}}\\
&\cA_3
}
\end{xy}
\]

More precisely, $R$ provides an isomorphism of $\RGtri $ with the open set $ \Gtetrop\subset \Gtetra$ of tetragonal pairs $(X,L)$ with the property that the $g^1_4$ $L$ is not special, i.e., $L$ contains no divisor of the form $2p+2q$ or $4p$ (see \cite[Theorem~2.9]{do_fibres}).
This open set dominates $\mathcal M_3$.

Note that $\Gtetra$ is clearly irreducible, since $\mathcal M_3$ and all the fibers of the projection $\Gtetra\to \mathcal M_3$ are so.
Therefore $\Gtetrop$ and $\RGtri$ are also irreducible. 

Now our purpose is to prove that Recillas' construction commutes with the natural involutions $\sigma$ and $\tau$ (defined in \eqref{eq:defsig} and \eqref{eq:deftau}), namely:

\begin{prop}\label{inv_commuten}
The equality $R\circ\sigma=\tau\circ R$ holds.
\end{prop}

By irreducibility, it is enough to check that $\sigma\circ R^{-1}=R^{-1}\circ\tau$ on an open set $U$.
We define $U$ to be the intersection of $\Gtetrop$ with the open set of pairs $(X,L)$ where $X$ is non-hyperelliptic and $L$ is not contained in the canonical bundle of $X$.

Hence let $X\in\cM_3$ be non-hyperelliptic, regarded as a quartic plane curve, and consider $L\in\Pic^4(X)\setminus\{\omega_X\}$ a complete $g^1_4$ on $X$ such that the linear system $|L|$ contains no divisor of the form $2p+2q$ or $4p$;
this is true for $L\in\Pic^4(X)\setminus\{\omega_X\}$ general. 

The element $R^{-1}(X,L)\in\RGtri$ can be explicitly described as follows: the curve
\[
C_L'=\{p+q\in X^{(2)}\mid h^0(L(-p-q))\neq0\}
\]
is smooth of genus $7$, with a fixed-point-free involution $i_L$ sending $p+q$ to the unique divisor $r+s\in|L(-p-q)|$.
The quotient $C_L=C_L'/\langle i_L\rangle$ has genus $4$ and naturally comes with a degree $3$ map to $\bP^1$, corresponding to a $g^1_3$ that we denote by $M_L$.
Then
\[
R^{-1}(X,L)=(C_L,\eta_L,M_L)
\]
where $\eta_L\in \left(JC_L\right)_2$ defines the \'etale cover $C_L'\to C_L$.

Now we denote $\widetilde{L}=\omega_X^2\otimes L^{-1}$, so that $\tau(X,L)=(X,\widetilde{L})$. 
Hence \autoref{inv_commuten} boils down to proving that
\[
(C_L,\eta_L,\omega_{C_L}\otimes M_L^{-1})=\sigma(R^{-1}(X,L))=R^{-1}(\tau(X,L))=(C_{\widetilde{L}},\eta_{\widetilde{L}},M_{\widetilde{L}}),
\]
which will be a consequence of the following two lemmas:
\begin{lem}
There is an isomorphism $C_{L}'\xrightarrow{\rho}C_{\widetilde{L}}'$ such that $\rho\circ i_L=i_{\widetilde{L}}\circ\rho$.
\end{lem}
\begin{proof}
For a point $p+q\in C_L'$, we let $\rho(p+q)\in X^{(2)}$ be the residual intersection of $X$ with the line $pq$ (this line being the tangent line to $X$ at $p$, if $p=q$).
Writing $r+s=i_L(p+q)$, then the eight points obtained by intersection of $X$ with the lines $pq$ and $rs$ give a divisor in $|\omega_X^2|$.
This implies $|\rho(p+q)+\rho(r+s)|\in|\widetilde{L}|$, which gives $\rho(p+q)\in C_{\widetilde{L}}'$ (i.e., $\rho$ is well-defined) and $i_{\widetilde{L}}(\rho(p+q))=\rho(r+s)=\rho(i_L(p+q))$.

To finish the proof, simply note that $\rho$ is an isomorphism since it has an obvious inverse.
\end{proof}

It follows that $(C_L,\eta_L)=(C_{\widetilde{L}},\eta_{\widetilde{L}})$ as elements of $\cR_4$.
To finish the proof of $\sigma(R^{-1}(X,L))=R^{-1}(\tau(X,L))$, we only need to check that the isomorphism $C_L\to C_{\widetilde{L}}$ induced by $\rho$ (that we denote by $\rho$ as well, abusing notation) sends $\omega_{C_L}\otimes M_L^{-1}$ to $M_{\widetilde{L}}$.

\begin{lem}\label{g13}
$\rho_*(\omega_{C_L}\otimes M_L^{-1})=M_{\widetilde{L}}.$
\end{lem}
\begin{proof}
Since $C_L$ (hence $C_{\widetilde{L}}$) is a non-hyperelliptic curve of genus $4$, by the discussion at the beginning of this subsection it turns out that $C_L$ has at most two $g^1_3$ (namely $M_L$ and $\omega_{C_L}\otimes M_L^{-1}$).
Therefore, it suffices to check that $\rho_*M_L\neq M_{\widetilde{L}}$.

Take $D=p_1+p_2+p_3+p_4\in|L|$. Then, $M_L$ is the line bundle on $C_L$ represented by the divisor
\[
\{p_1+p_2,p_3+p_4\}+\{p_1+p_3,p_2+p_4\}+\{p_1+p_4,p_2+p_3\}
\]
and therefore $\rho_*M_L$ on $C_{\widetilde{L}}$ is represented by the divisor
\[
\{a_{12}+b_{12},a_{34}+b_{34}\}+\{a_{13}+b_{13},a_{24}+b_{24}\}+\{a_{14}+b_{14},a_{23}+b_{23}\}
\]
where for $i,j\in\{1,2,3,4\}$ the points $a_{ij},b_{ij}$ are the residual intersection of $X$ with the line $p_ip_j$.
We may take $D$ with $p_1,p_2,p_3,p_4$ distinct, and such that for every $i\neq j$ we have $a_{ij}\neq b_{ij}$.

If the equality $\rho_*M_L=M_{\widetilde{L}}$ were true, then we would have equalities in $X^{(4)}$
\[
a_{12}+b_{12}+a_{34}+b_{34}=a_{13}+b_{13}+a_{24}+b_{24}=a_{14}+b_{14}+a_{23}+b_{23}
\]
with this divisor representing the line bundle $\widetilde{L}$.
But these equalities are easily seen to imply that the points $p_1,p_2,p_3,p_4$ are collinear, which is a contradiction since $L\neq\omega_X$.
\end{proof}

This finishes the proof of \autoref{inv_commuten}.
As a consequence of it, we deduce that $\tau$ leaves invariant the image of $R$, and the fixed points of the two involutions correspond by $R$.

First, let us study the fixed points of $\sigma $.
If $(C,\eta,M)\in \RGtri$ with $C$ non-hyperelliptic, then $M\cong \omega_C\otimes M^{-1}$ if and only if $M$ is the unique $g^1_3$ on $C$, namely $C\in \cT_4$ and $M$ is the semicanonical pencil of $C$.
Therefore, the locus of fixed points of $\sigma$ consists of two pieces:
\begin{enumerate}
 \item [\mylabel{a1}{(a1)}] Triples $(C,\eta,M)\in\RGtri$ with $(C,\eta)\in(\cT^e_4\cup\cT^o_4)\setminus\cR\cH_4$ and $M$ a semicanonical pencil on $C$.
 
 \item [\mylabel{a2}{(a2)}] The set of $(C,\eta,M)\in\RGtri$ with $(C,\eta)\notin\cT^e_4\cup\cT^o_4$ having a nontrivial automorphism exchanging the two $g^1_3$'s on $C$.

\end{enumerate}

The locus of fixed points of $\tau$ is formed by the following pieces: 

\begin{enumerate}
 \item [\mylabel{b1}{(b1)}] The $\mathbb P^2$-bundle on $\cM_3$ given by the union of the exceptional divisors in $\cG^1_4(X)$ moving $X$, that is: $\underset{X\in \cM_3}{\bigcup}\left(\vert\omega_X\vert^*\setminus X\right)/\Aut(X)$.
 \item [\mylabel{b2}{(b2)}] Pairs $(X,L)\in\Gtetra$ with $X$ hyperelliptic.
 Indeed, the hyperelliptic involution on $X$ exchanges any $L\in\Pic^4(X)$ with $\omega_X^2\otimes L^{-1}$.
 
 \item [\mylabel{b3}{(b3)}] The set of pairs $(X,\omega_X\otimes\eta)$ with $X\in\cM_3$ and $\eta\in JX_2\setminus\{\cO_X\}$. This set is naturally identified with $\cR_3$.
 \item [\mylabel{b4}{(b4)}] The (closure of) the set of pairs $(X,L)\in\Gtetra$ with $X$ non-hyperelliptic, having a nontrivial automorphism sending $L$ to $\omega_X^2\otimes L^{-1}$.
\end{enumerate}


\begin{rem}
The piece \ref{b3} corresponds, under $R^{-1}$, to the irreducible component of \ref{a2} formed by covers of bielliptic curves of genus $4$.

Indeed, consider $(X,L)$ with $X$ a quartic plane curve and $L=\omega_X\otimes\eta$, $\eta\in JX_2\setminus\{\cO_X\}$.
We can express $L=\theta_1\otimes\theta_2$ for two distinct odd theta-characteristics $\theta_1$ and $\theta_2$;
namely, $|L|$ has a divisor given by the contact points of two distinct bitangent lines.
Moreover, using the theory of syzygetic triads (\cite[Section~5.4.1]{dolg}) it is easy to check that $|L|$ contains exactly six such ``distinguished'' divisors (i.e., formed by contact points of two bitangent lines).

Then, the curves $C_L$ and $C_{\widetilde{L}}$ are equal by definition.
Moreover, the involution $\rho:C_L\to C_L$ has exactly six fixed points, lying over the six distinguished divisors of $|L|$.
It follows that $C_L$ is bielliptic, and the bielliptic involution $\rho$ exchanges the two $g^1_3$'s on $C_L$ by \autoref{g13}.

Conversely, Recillas' trigonal construction applied to a cover of a bielliptic curve is well known to give an element of \ref{b3}, see \cite[Section~3]{dolg2}.
\end{rem}

\vspace{0.15mm}

Keeping all this in mind, we are now ready to prove that $\cP_4\, \vert_{\cT^o_4}$ is surjective, and show that its fiber over a general Jacobian $JX\in\cA_3$ is the complement in the projective plane of the union of the canonical model of $X$ and the $28$ lines that are bitangent to it.
Moreover, we conclude that $\cT^0_4$ is irreducible and rationally connected.

\begin{proof}[Proof of \autoref{mainB}.\eqref{mainB:g4}]
We will first prove that $\cP_4\, \vert_{\cT^o_4}$ is dominant, by describing the fiber of a general Jacobian $JX\in\cA_3$ (in particular, showing its non-emptiness).
To this end, we take a non-hyperelliptic curve $X\in\cM_3$ without automorphisms, and denote $\cG^{1,ns}_4(X)=\psi^{-1}(X)\cap\Gtetrop$. 
That is, $\cG^{1,ns}_4(X)$ parametrizes $g^1_4$ linear series on $X$ with no divisor of the form $2p+2q$ or $4p$.

Observe that, since the whole fiber $\cP_4^{-1}(JX)$ is contained in $\cR_4\setminus\cR\cH_4$, according to Recillas' diagram, the fiber $\cP_4\, \vert_{\cT_4^o}^{-1}(JX)$ equals $\varphi\big(R^{-1}(cG^{1,ns}_4(X))\big)\cap\cT^o_4$.
The latter is isomorphic to $R^{-1}\big(\cG^{1,ns}_4(X)\big)\cap\varphi^{-1}(\cT^o_4)$, since the restriction $\varphi\, \vert_{\varphi^{-1}(\cT^o_4)}$ is an isomorphism;
note that the intersection $R^{-1}\big(\cG^{1,ns}_4(X)\big)\cap\varphi^{-1}(\cT^o_4)$ lies in the piece \ref{a1} of the locus of fixed points of $\sigma$.

On the other hand, by our assumptions on $X$, the intersection of $\cG^{1,ns}_4(X)$ with the locus of fixed points of $\tau$ consists of a $2$-dimensional irreducible component (intersection with the piece \ref{b1}) and finitely many points (intersection with the piece \ref{b3}).
Therefore, the intersection of $R^{-1}\big(\cG^{1,ns}_4(X)\big)$ with the locus of fixed points of $\sigma$ consists of finitely many points of the piece \ref{a2} and a $2$-dimensional component.

We claim that this $2$-dimensional component must be $R^{-1}\big(\cG^{1,ns}_4(X)\big)\cap\varphi^{-1}(\cT^o_4)$. 
Indeed, if this were not the case then the piece \ref{b1} would correspond to \ref{a2};
hence \ref{a2} would be $8$-dimensional, which is absurd since the locus of (non-hyperelliptic) curves of genus $4$ with automorphisms is well known to have lower dimension (see \cite{corn}).

All in all, we obtain that the fiber is isomorphic to the intersection of $\cG^{1,ns}_4(X)$ with the locus \ref{b1} of the fixed locus of $\tau$.
This intersection is the set of all the non-complete, base-point-free $g^1_4$ on $X$ containing no divisor of the form $2p+2q$ or $4p$.
Such a $g^1_4$ is defined by the pencil of lines through a fixed point of $\bP^2$, outside the curve $X$ and lying in no bitangent line to $X$.

Now we proceed to prove that the map $\cP_4\, \vert_{\cT^o_4}:\cT_4^o\to\cA_3$ is not only dominant, but also surjective.
For this, note that we have shown that $\varphi^{-1}(\cT^o_4)$ is mapped via $R$ to the locus \ref{b1};
it follows that Jacobians $JX\in\cA_3$ of non-hyperelliptic curves $X$ with automorphisms lie in the image of $\cP_4\, \vert_{\cT^o_4}$ as well. 
In addition, every element of $\theta_{null}$ (i.e., a hyperelliptic Jacobian or product of Jacobians in $\cA_3$) can be obtained as the Prym variety of a cover in $\cR\cH_4\subset\cT^o_4$;
this follows from Mumford's description of Prym varieties of covers of hyperelliptic curves (\cite[Page~346]{mu}), that we already used in \autoref{sec:Te}.

Therefore, we have proved that $\cP_4\, \vert_{\cT^o_4}:\cT_4^o\to\cA_3$ is surjective, with all the fibers of elements in $\cA_3\setminus\theta_{null}$ being irreducible of the same dimension.
Moreover, since $\cP_4^{-1}(\theta_{null})=\cT^e_4$ by \autoref{thetanull2}.\eqref{thetanull2:from3}, we have $\cP_4\, \vert_{\cT^o_4}^{-1}(\theta_{null})=\cT^o_4\cap\cT^e_4=\cR\cH_4$.

Thus if $\cT^o_4$ were not irreducible, it would have $\cR\cH_4$ as an irreducible component, contradicting the equidimensionality of $\cT^o_4$.

Finally, the rational connectedness of $\cT^o_4$ follows again from the results in \cite{ghs}.
\end{proof}


We end this section by showing that the fiber $\cP_4^{-1}(JX)\subset\cT^e_4$ of a general hyperelliptic Jacobian $JX\in\theta_{null}$ is birationally equivalent to its Kummer variety.
As a consequence, we get that $\cT^e_4$ is irreducible.

\begin{proof}[Proof of \autoref{mainA}.\eqref{mainA:fiber}]
Take a general hyperelliptic curve $X\in\cH_3$ (in particular, having the hyperelliptic involution as its only nontrivial automorphism).
The intersection $\cP_4^{-1}(JX)\cap\cR\cH_4$ can be described following Mumford's trick for covers of hyperelliptic curves, as we did in the case of genus $3$: this intersection is the complement in $\bP^2$ of the union of eight lines and a conic.

Now, we proceed to describe the ``non-hyperelliptic'' part of the fiber $\cP_4^{-1}(JX)$.
As usual, denote by $\cG^{1,ns}_4(X)=\psi^{-1}(X)\cap\Gtetrop$; due to the action of the hyperelliptic involution on $\cG^1_4(X)$, $\cG^{1,ns}_4(X)$ is birationally equivalent to the Kummer variety of $JX$.

According to the commutative diagram given by Recillas' construction, $\cP_4^{-1}(JX)\setminus\cR\cH_4$ is $\varphi(R^{-1}(\cG^{1,ns}_4(X)))$, which is isomorphic to $R^{-1}(\cG^{1,ns}_4(X))$ since the restriction $\varphi\, \vert_{\varphi^{-1}(\cT^e_4)}$ is an isomorphism.
It follows that $\cP_4^{-1}(JX)\setminus\cR\cH_4$ is birational to the Kummer variety of $JX$.

Finally, we point out that the whole fiber $\cP_4^{-1}(JX)$ is irreducible.
This follows from our study of $\cP_4^{-1}(JX)\setminus\cR\cH_4$ and $\cP_4^{-1}(JX)\cap\cR\cH_4$, as well as from the equidimensionality of $\cT^e_4$.
\end{proof}




\section{Genus 5 and cubic threefolds}\label{sec:genus5}

Similarly to the cases $g=3$ and $g=4$, the behavior of the Prym map on $\cT^o_5$ is quite different from its behavior on $\cT^e_5$;
indeed, $\cT^o_5$ dominates $\cA_4$ as already observed by Izadi in \cite[Proof of Theorem~6.14]{iz}.
In this section we first give a brief different proof of this fact, by means of the cohomology classes of $\cT^o_5$ and $\cT^e_5$ (\autoref{To5domin});
then we study in more detail the generic fiber of $\tcP_5|_{\cT^o_5}$, whose geometry reveals enumerative properties of cubic threefolds.

We start by recalling Donagi's description \cite[Section~5]{do_fibres} of the generic fiber of the (proper, surjective) Prym map $\tcP_5:\tcR_5\to\cA_4$.
There is a birational map
\[
\kappa:\cA_4\dasharrow\cR\cC^+
\]
where $\cR\cC^+$ denotes the moduli space of pairs $(V,\delta)$ with $V\subset\bP^4$ a smooth cubic threefold and $\delta\in JV_2$ an even $2$-torsion point (i.e., $\delta\notin\Theta_V$ for the canonical choice of the theta divisor $\Theta_V\subset JV$).
Izadi \cite{iz} explicitly describes an open subset of $\cA_4$ on which $\kappa$ is an isomorphism.

Then the fiber of $\kappa\circ\tcP_5$ over a generic $(V,\delta)$ is isomorphic to the double \'etale cover $\widetilde{F(V)}$ of the Fano surface of lines $F(V)$ defined by $\delta$ (recall that $\Pic^0(F(V))\cong JV$).

\begin{prop}\label{To5domin}
The restricted Prym map $\tcP_5|_{\cT^o_5}:\cT^o_5\to\cA_4$ is dominant.
\end{prop}
\begin{proof}
For $A\in\cA_4$ general, we write $(V,\delta)=\kappa(A)$ and let $\widetilde{F(V)}=(\tcP_5)^{-1}(A)$ be its fiber by $\tcP_5$.
If $\iota:\widetilde{F(V)}\hookrightarrow\tcR_5\hookrightarrow\ocR_5$ denotes the inclusion, then the pullback map
\[
\iota^*:\Pic(\ocR_5)_\bQ\to\Pic(\widetilde{F(V)})_\bQ
\]
annihilates the classes $\delta_0',\delta_0'',\delta_1,\delta_2,\delta_3,\delta_4,\delta_{1:4}$ and $\delta_{2:3}$.

Indeed, the generic element of the divisors $\Delta_i$ ($i=1,\ldots,4$) and $\Delta_{i:5-i}$ ($i=1,2$) is an admissible cover whose Prym variety is a decomposable ppav;
hence $\tcP_5(\Delta_{i}\cap\tcR_5)$ and $\tcP_5(\Delta_{i:5-i}\cap\tcR_5)$ have positive codimension in $\cA_4$, so $\Delta_{i}\cap\tcR_5$ and $\Delta_{i:5-i}\cap\tcR_5$ are disjoint with the general fiber $\widetilde{F(V)}$.
On the other hand $\Delta ''_0$ is formed by Wirtinger covers, so $\tcP_5(\Delta_{0}''\cap\tcR_5)$ is contained in the Jacobian locus, which is a divisor in $\cA_4$;
this again implies that $\iota^*(\delta''_0)=0$. 
Finally, $\iota^*(\delta'_0)=0$ since $\widetilde{F(V)}\subset\tcR_5$ and the general element of $\Delta_0'$ is a non-admissible cover.

Since $\cT^e_5$ is mapped by $\tcP_5$ to $\theta_{null}\subset\cA_4$, the divisor $\cT^e_5$ does not intersect the generic fiber $\widetilde{F(V)}$.
Hence, by \autoref{class}, we obtain
\[
0=\iota^*[\cT^e_5]=68\iota^*\lambda-17\iota^*\delta_0^{ram}
\]
which implies the relation $\iota^*\delta_0^{ram}=4\iota^*\lambda$.

Now, the restriction of $\cT^o_5$ to the fiber $\widetilde{F(V)}$ has cohomology class
\[
\iota^*[\cT^o_5]=64\iota^*\lambda-15\iota^*\delta_0^{ram}=4\iota^*\lambda,
\]
which is clearly nonzero since the Hodge structure cannot remain constant along the (open) subset of $\widetilde{F(V)}$ formed by smooth covers.
Therefore, the restriction of $\cT^o_5$ to the generic fiber is not trivial and $\cT^o_5$ dominates $\cA_4$.
\end{proof}

In the previous proof, observe that the class $\iota^*[\cT^o_5]$ equals $\iota^*\delta_0^{ram}$.
This is consistent with the fact that the involution $j:\widetilde{F(V)}\to\widetilde{F(V)}$ (induced by the double \'etale cover) exchanges $\cT^o_5\cap\widetilde{F(V)}$ and $\Delta_0^{ram}\cap\widetilde{F(V)}$.
In order to understand this, we need to recall the geometric description of $j$ given also by Donagi in \cite[Section~5]{do_fibres}.

Given a smooth cover $(C,\eta)\in\widetilde{F(V)}$, there exists a unique representation $JC=\tcP_6(Q,\nu)$ as the Prym of a cover of a plane quintic $Q$ (the theta-characteristic $\cO_Q(1)\otimes\nu$ on $Q$ being even).
More explicitly, $(Q,\nu)$ is the double cover induced by the involution $-1_{JC}$ on the symmetric curve $W^1_4(C)\subset JC$.

Consider the short exact sequence
\[
0\to\langle\nu\rangle\to \langle\nu\rangle^\perp\to JC_2\to 0
\]
where $\langle\nu\rangle^\perp\subset JQ_2$ is the orthogonal for the Weil pairing on $JQ_2$.
Then the preimage of $\langle\eta\rangle\subset JC_2$ is a totally isotropic subgroup of four elements $0,\nu,\sigma$ and $\sigma\otimes\nu$;
moreover, the theta-characteristics $\cO_Q(1)\otimes\sigma$ and $\cO_Q(1)\otimes\sigma\otimes\nu$ on $Q$ have opposite parities.

Say $\cO_Q(1)\otimes\sigma\otimes\nu$ is even. Then $\tcP_6(Q,\sigma\otimes\nu)$ is the Jacobian of a genus 5 curve $C'$, and $\nu\in\langle\sigma\otimes\nu\rangle^\perp\subset JQ_2$ induces a nonzero element $\eta'\in (JC')_2$;
one has $j(C,\eta)=(C',\eta')$. 

This picture beautifully closes with the observation that $\tcP_6(Q,\sigma)\cong JV$ as ppav and the even $2$-torsion point $\delta\in JV_2$ is induced by $\nu\in \langle\sigma\rangle^\perp\subset JQ_2$.
In particular, the double cover $\widetilde{F(V)}\to F(V)$ sends $(C,\eta)$ to the line $l\in F(V)$ having $Q$ as discriminant curve for the conic bundle structure defined by $l$.

\begin{rem}\label{lambdaizadi}
In \cite[Section~3]{iz}, Izadi gives an alternative realization of the involution $j$.
Given a smooth cover $(\tC,C)=(C,\eta)\in\cR_5$, the theta-dual $T(\tC)=V^2(C,\eta)\subset P^-(C,\eta)$ is a symmetric curve (when properly translated to $P(C,\eta)$).
If $C'$ is the quotient of $T(\tC)$ by $-1$, then the cover $(T(\tC),C')$ corresponds to $(\tC,C)$ under the involution $j$.
\end{rem}

For any smooth cubic threefold $V\subset\bP^4$, we consider the set
\[
\Gamma=\{l\in F(V)\mid \exists\text{ a $2$-plane $\pi$ and a line $r\in F(V)$ with }V\cdot\pi=l+2r\}
\]
parametrizing the lines $l\in F(V)$ whose conic bundle structure has a singular discriminant curve (\cite[Proposition~1.2]{be_jacob}). 
In other words, $\Gamma$ parametrizes presentations of $JV$ as the Prym variety of an (admissible) cover of a singular (plane quintic) curve.

The set $\Gamma$ is well known to have pure dimension $1$ and it is irreducible for a general cubic threefold $V$;
assuming the irreducibility of $\Gamma$ (which will be proven in \autoref{gamma}), let us determine the behavior of the general fiber of $\tcP_5|_{\cT^o_5}$ under the involution $j$ of $\widetilde{F(V)}$.

\begin{prop}\label{involution}
For a generic $(V,\delta)\in\cR\cC^+$, the preimage $\widetilde{\Gamma}\subset\widetilde{F(V)}$ of $\Gamma$ by the double \'etale cover consists of two irreducible components, namely $\cT^o_5\cap\widetilde{F(V)}$ and $\Delta^0_{ram}\cap\widetilde{F(V)}$.
These two components are exchanged by the involution $j$.
\end{prop}
\begin{proof}
We will see in \autoref{gamma} (note that the first paragraph of its proof is independent of this proposition) that $\Gamma$ is irreducible. 
Hence, $\widetilde{\Gamma}$ has at most two irreducible components.
Thus, for the first statement, it suffices to check that both $\cT^o_5\cap\widetilde{F(V)}$ and $\Delta^0_{ram}\cap\widetilde{F(V)}$ are contained in $\widetilde{\Gamma}$.

On the one hand, if $(C,\eta)\in\cT^o_5\cap\widetilde{F(V)}$ with $C$ smooth, then the associated quintic $Q=W^1_4(C)/\langle\pm1\rangle$ is singular;
indeed, the semicanonical pencil on $C$ is a point of $W^1_4(C)$ fixed by the involution.

On the other hand, given a general element $(C,\eta)\in\Delta^0_{ram}$ the expression of the (semi-abelian) variety $JC$ as a Prym variety necessarily comes from a non-admissible cover $(Q,\sigma)$.
In particular, $Q$ is singular as well.

The fact that $j$ exchanges the components of $\widetilde{\Gamma}$ is nothing but \cite[Lemma~3.14]{iz}.
This is also immediately observed from \autoref{lambdaizadi}, and the fact that for smooth covers $(\tC,C)=(C,\eta)\in\cT^o_5$ the theta-dual $T(\tC)=V^2(C,\eta)$ is singular by \autoref{V2}.
\end{proof}

It follows that the general fiber of $\tcP_5|_{\cT^o_5}:\cT^o_5\to\cA_4$ is a curve dominating $\Gamma$, which concludes the proof of \autoref{mainB}.\eqref{mainB:g5}.

In the rest of this section, we concentrate on the geometry of both curves $\Gamma$ and $\cT^o_5\cap\widetilde{F(V)}$; in particular, \autoref{DesingPrecisa} will provide a more precise description of the partial desingularization of $\Gamma$ appearing in \autoref{mainB}.\eqref{mainB:g5}.
To this end, we consider for any smooth cubic threefold $V\subset\bP^4$ the curve
\[
\Gamma'=\{r\in F(V)\mid \exists\text{ a $2$-plane $\pi$ and a line $l\in F(V)$ with }V\cdot\pi=l+2r\}
\]
formed by \emph{lines of second type}.
In contrast to $\Gamma$, this curve has received considerable attention in the literature.
For instance:
\begin{enumerate}[(1)]
 \item (\cite[Proposition~10.21]{cg}) $\Gamma'$ has pure dimension 1 and, as a divisor in the Fano surface $F(V)$, is linearly equivalent to twice the canonical divisor $K_{F(V)}$.
 
 \item In \cite[Corollary~1.9]{murre}, it is stated that $\Gamma'$ is smooth for every $V$.
 In fact, what Murre's local computations really show is that the singularities of $\Gamma'$ correspond to lines $r\in F(V)$ for which there exists a $2$-plane $\pi$ satisfying $V\cdot\pi=3r$.\footnote{For the interested reader, the linear polynomial $l$ appearing in equation (13) of \cite[Page~167]{murre} not only depends on the variables $u$ and $v$, but also on the variable $x$ (see also \cite{BB22}).}
 
 \noindent An easy count of parameters shows that, for a general $V$, such lines do not exist (namely the curve $\Gamma'$ is smooth).
\end{enumerate}

For every $r\in\Gamma'$, there exist a unique $2$-plane $\pi$ and a unique $l\in F(V)$ such that $V\cdot\pi=l+2r$ (see \cite[Lemma~2.4]{naranjo-ortega}). This naturally defines a surjective morphism $\Gamma'\to\Gamma$.
Even if one would be tempted to think that it defines an isomorphism between $\Gamma'$ and $\Gamma$, this is not the case\footnote{In particular, this fixes an inaccuracy in the original published version of \cite[Proposition~2.6]{naranjo-ortega}, already corrected in the arXiv version \href{https://arxiv.org/abs/1708.06512}{arXiv:1708.06512.v3}.}:

\begin{prop}\label{gamma}
For a general smooth cubic threefold $V$, the following hold:
\begin{enumerate}[\rm (1)]
 \item $\Gamma'$ is smooth and irreducible.
 \item $\Gamma$ is irreducible and singular, with only nodes as singularities.
 \item The map $\Gamma'\to\Gamma$ is birational (i.e., $\Gamma'$ is the normalization of $\Gamma$).
\end{enumerate}
\end{prop}
\begin{proof}
The smoothness of $\Gamma'$ being known by the discussion above, for the irreducibility of $\Gamma'$ (and hence that of $\Gamma$) one argues as in the proof of \cite[Proposition~2.6]{naranjo-ortega}. 
In particular, the irreducibility of $\Gamma$ completes the proof of \autoref{involution}.

In order to prove that $\Gamma'\to\Gamma$ is birational, we need to prove that a general $l\in\Gamma$ has a unique preimage in $\Gamma'$.
Note that the preimages of a line $l\in\Gamma$ correspond to nodes on the discriminant (plane quintic) curve $Q_l$ of the conic bundle structure defined by $l$.

Fix an even $2$-torsion point $\delta\in JV_2$ (such that the pair $(V,\delta)\in\cR\cC^+$ is general), and denote by $\varphi:\widetilde{F(V)}\to F(V)$ the \'etale double cover defined by $\delta$.
If $(C,\eta)\in\tcR_5$ denotes a Prym curve lying in both $\cT^o_5$ and $\varphi^{-1}(l)$, then by \autoref{involution} the nodes of $Q_l$ are also in correspondence with the semicanonical pencils on the curve $C\in\cT_5$.
Since the general curve of $\cT_5$ has a unique semicanonical pencil, the birationality of $\Gamma'\to\Gamma$ follows.


Now we proceed to prove that the curve $\Gamma$ is always singular.
For this, it suffices to check that there exist points of $\Gamma$ with (at least) two preimages in $\Gamma'$. Namely, that there exist lines $l\subset V$ such that there are two distinct $2$-planes $\pi_1,\pi_2$ and lines $r_1,r_2\subset V$ with the property $V\cdot\pi_i=l+2r_i$ ($i=1,2$).

Take a reference system in $\bP^4$ so that $l\cap r_1=[0:0:0:1:0]$, $l\cap r_2=[0:0:0:0:1]$, $[0:1:0:0:0]\in r_1$ and $[0:0:1:0:0]\in r_2$.
Denoting by $x,y,z,u,v$ the homogeneous coordinates in this reference system, a cubic threefold $V$ will satisfy $V\cdot\pi_i=l+2r_i$ ($i=1,2$) if and only if it admits an equation of the form
\[
F(x,y,z,u,v)=x\cdot Q(x,y,z,u,v)+\lambda_{15}y^2z+\lambda_{16}yz^2+\lambda_{17}yzu+\lambda_{18}yzv+\lambda_{19}yv^2+\lambda_{20}zu^2
\]
with $Q$ a quadratic polynomial.
This family of equations forms a 20-dimensional linear variety in the projective space $\bP^{34}$ of cubic equations in five variables (up to constant).

On the other hand, the projective transformations which leave the lines $l,r_1,r_2$ invariant depend on ten projective parameters.
Therefore, the moduli space of cubic threefolds $V$ for which there exist lines $l,r_1,r_2\subset V$ and $2$-planes $\pi_1,\pi_2$ as asserted is $10$-dimensional.
In other words, every smooth cubic threefold $V$ admits such a configuration.

A similar parameter count shows that for a general cubic threefold $V$:
\begin{itemize}[\textbullet]
 \item There are no lines $l\in\Gamma$ admitting three or more preimages in $\Gamma'$.
 \item There are no lines $l\in\Gamma$ admitting two distinct preimages $r_1,r_2\in\Gamma'$ with the property $l\cap r_1=l\cap r_2$.
\end{itemize}

This shows that, for a general cubic threefold $V$, the curve $\Gamma$ is singular and its singular points are of multiplicity $2$.
Hence to finish the proof, it only remains to check that such singular points are ordinary.
We will prove the following: for any singular point $l\in\Gamma$, the tangent directions to $\Gamma'$ at the two preimages $r_1,r_2\in\Gamma'$ of $l$ are mapped to independent directions in the tangent space to $F(V)$ at $l$.
We will use the local analysis of $\Gamma'$ performed by Murre (\cite[Section~1A]{murre}).

According to our discussion, the singularities of $\Gamma'$ correspond to lines $l\subset V$ for which there exist two distinct $2$-planes $\pi_1,\pi_2$ and (disjoint) lines $r_1,r_2\subset V$ such that $V\cdot\pi_i=l+2r_i$ ($i=1,2$). 
Taking coordinates as before, we let
\begin{multline*}
 F(x,y,z,u,v)=\lambda_{0}x^3+\lambda_{1}x^2y+\lambda_{2}x^2z+\lambda_{3}x^2u+\lambda_{4}x^2v+\lambda_{5}xy^2+\lambda_{6}xyz+\lambda_{7}xyu+\lambda_{8}xyv+\lambda_{9}xz^2 \\
 +\lambda_{10}xzu+\lambda_{11}xzv+\lambda_{12}xu^2+\lambda_{13}xuv+\lambda_{14}xv^2+\lambda_{15}y^2z+\lambda_{16}yz^2+\lambda_{17}yzu+\lambda_{18}yzv+\lambda_{19}yv^2+\lambda_{20}zu^2
 \end{multline*}
be the equation defining $V$.
Observe that $\lambda_{19}\neq0$ and $\lambda_{20}\neq0$, otherwise $V$ would contain one of the $2$-planes $\pi_1:x=z=0,\;\pi_2:x=y=0$ and hence $V$ would be singular.

In the Grassmannian $\mathbb{G}(1,4)$ of lines in $\bP^4$, we take local coordinates $x',x'',y',y'',z',z''$ for lines $l'$ around $l$, given by
\[
l'\cap\{v=0\}=[x':y':z':1:0],\;\;l'\cap\{u=0\}=[x'':y'':z'':0:1].
\]

Since $F$ can be written as $F=x\cdot f(x,y,z,u,v)+y\cdot g(x,y,z,u,v)+z\cdot h(x.y,z,u,v)$ with
\begin{align*}
 f(x,y,z,u,v)&=\lambda_{12}u^2+\lambda_{13}uv+\lambda_{14}v^2+\text{terms of lower degree in $u,v$}, \\
 g(x,y,z,u,v)&=\lambda_{19}v^2+\text{terms of lower degree in $u,v$}, \\
 h(x,y,z,u,v)&=\lambda_{20}u^2+\text{terms of lower degree in $u,v$},
\end{align*}
following \cite[Section~1A]{murre} the tangent plane $T_lF(V)$ to $F(V)$ at $l$ is described by the four independent equations
\[
\lambda_{12}x'+\lambda_{20}z'=\lambda_{13}x'+\lambda_{12}x''+\lambda_{20}z''=\lambda_{14}x'+\lambda_{13}x''+\lambda_{19}y'=\lambda_{14}x''+\lambda_{19}y''=0.
\]
Since $\lambda_{19},\lambda_{20}\neq0$, observe that the coordinates $x',x''$ are independent in this tangent plane.

Similarly, we take local coordinates $a',a'',b',b'',c',c''$ for lines $r'$ around $r_1$ in $\mathbb{G}(1,4)$, where
\[
r'\cap\{u=0\}=[a':1:b':0:c'],\;\;r'\cap\{y=0\}=[a'':0:b'':1:c''].
\]
Following again \cite[Section~1A]{murre}, $T_{r_1}F(V)$ is described by the independent equations
\[
\lambda_{5}a'+\lambda_{15}b'=\lambda_{7}a'+\lambda_{5}a''+\lambda_{17}b'+\lambda_{15}b''=\lambda_{12}a'+\lambda_{7}a''+\lambda_{20}b'+\lambda_{17}b''=\lambda_{12}a''+\lambda_{20}b''=0,
\]
which are equivalent to $a'=a''=b'=b''=0$.
Therefore, we may take $c',c''$ as coordinates in the tangent plane $T_{r_1}F(V)$, which is naturally identified with the set of lines contained in the $2$-plane $\pi_1$ (and avoiding the point $[0:0:0:0:1]\in\pi_1$).

Under our assumptions of generality on $V$, the analysis in \cite[Section~1A]{murre} shows that $\Gamma'$ is smooth at $r_1$, with tangent line
\[
T_{r_1}(\Gamma'):(\lambda_7\lambda_{20}-\lambda_{12}\lambda_{17})c'+(\lambda_{12}\lambda_{15}-\lambda_5\lambda_{20})c''=0
\]
(again, $\lambda_7\lambda_{20}-\lambda_{12}\lambda_{17}$ and $\lambda_{12}\lambda_{15}-\lambda_5\lambda_{20}$ are not simultaneously zero by the smoothness of $V$).

Let us assume that $\lambda_7\lambda_{20}-\lambda_{12}\lambda_{17}\neq0$. Given $c''\in\bC$, we denote by $r_{1,c''}\in T_{r_1}(\Gamma')$ the line
\[
r_{1,c''}=\Big[0:1:0:0:\frac{\lambda_5\lambda_{20}-\lambda_{12}\lambda_{15}}{\lambda_7\lambda_{20}-\lambda_{12}\lambda_{17}}c''\Big]\vee[0:0:0:1:c''].
\]

Using the description in \cite[1.3]{murre}, elementary (but tedious) calculations show that the first order deformation $r_{1,c''}$ of $r_1$ along $\Gamma'$ yields a first order deformation of the $2$-plane $\pi_1$ given by
\[
\pi_{1,c''}=r_{1,c''}\vee\Big[\frac{\lambda_{19}\lambda_{20}}{\lambda_7\lambda_{20}-\lambda_{12}\lambda_{17}}c'':0:-\frac{\lambda_{12}\lambda_{19}}{\lambda_7\lambda_{20}-\lambda_{12}\lambda_{17}}c'':0:1\Big],
\]
and thus a first order deformation $l_{c''}$ of $l$ along $\Gamma$ given by the following local coordinates around $l$:
\begin{align*}
 x'(l_{c''})=0,\;\;\;y'(l_{c''})&=-\frac{\lambda_{13}\lambda_{20}}{\lambda_7\lambda_{20}-\lambda_{12}\lambda_{17}}c'',\;\;\;z'(l_{c''})=0,\\
 x''(l_{c''})=\frac{\lambda_{19}\lambda_{20}}{\lambda_7\lambda_{20}-\lambda_{12}\lambda_{17}}c'',\;\;\;y''(l_{c''})&=-\frac{\lambda_{14}\lambda_{20}}{\lambda_7\lambda_{20}-\lambda_{12}\lambda_{17}}c'',\;\;\;z''(l_{c''})=-\frac{\lambda_{12}\lambda_{19}}{\lambda_7\lambda_{20}-\lambda_{12}\lambda_{17}}c''.
\end{align*}

In other words, the first order deformation of $r_1$ along $\Gamma'$ defines the tangent direction $x'=0$ to $\Gamma$ at $l$.
A similar analysis shows that the first order deformation of $r_2$ along $\Gamma'$ induces the tangent direction $x''=0$ to $\Gamma$ at $l$.
Since these directions are distinct, it follows that $l$ is a node of the curve $\Gamma$, which finishes the proof.
\end{proof}

In view of the existence of singularities of $\Gamma$ and their geometric significance, it seems a natural enumerative question to ask about the number of nodes of the curve $\Gamma$ for a general cubic threefold $V$.

Let us recall that for $V$ very general, the Fano surface $F(V)$ has Picard number $1$ (\cite[Section~1.3]{roulleau}).
The (numerical) Néron-Severi group $\NS(F(V))\cong\bZ$ is generated by the class $L$ of the incidence divisor
\[
C_s=\text{(closure of) }\{l\in F(V)\mid l\cap s\neq\emptyset,\;l\neq s\}\subset F(V)
\]
of lines intersecting a given line $s\in F(V)$ (note that $C_s\subset JV$ is the Abel-Prym curve for the Prym presentation of $JV$ using the conic bundle structure defined by $s$).

We have $L^2=5$, as this is the number of lines on a smooth cubic surface intersecting two given skew lines on it. 
Moreover, $K_{F(V)}=3L$ in $\NS(F(V))$ and $\Gamma'$ is linearly equivalent to $2K_{F(V)}$ (see \cite[Section~10]{cg}).

\begin{thm}\label{classgamma}
For every smooth cubic threefold $V$, the curve $\Gamma$ is numerically equivalent to $8K_{F(V)}$ in $F(V)$.
\end{thm}
\begin{proof}
Since both $\Gamma$ and $K_{F(V)}$ are the restriction to $F(V)$ of divisors in the universal Fano variety of lines in cubic threefolds, it is enough to prove the result when $V$ is very general.

Pick an even $2$-torsion point $\delta\in JV_2$, and consider the \'etale double cover $\varphi:\widetilde{F(V)}\to F(V)$ defined by $\delta$.
Recall from \autoref{involution} that $\widetilde{\Gamma}=\varphi^{-1}(\Gamma)$ has $\cT^o_5\cap\widetilde{F(V)}$ and $\Delta_0^{ram}\cap\widetilde{F(V)}$ as irreducible components (exchanged by the natural involution on $\widetilde{F(V)}$).

Let us write $mL$ for the class of $\Gamma$ in $\NS(F(V))$.
Recall from the proof of \autoref{To5domin} that, if $\iota^*:\Pic(\ocR_5)_\bQ\to\Pic(\widetilde{F(V)})_\bQ$ is the natural pullback map, then $\iota^*([\cT^o_5])=\iota^*(\delta^0_{ram})=4\iota^*\lambda$ and $\iota^*$ annihilates any other basic divisor class of $\Pic(\ocR_5)_\bQ$.
Therefore, we have an equality
\[
m\varphi^*L=\varphi^*\Gamma=8\iota^*\lambda
\]
in $\NS(\widetilde{F(V)})_\bQ$.
Now we will compare the classes $\varphi^*L$ and $\iota^*\lambda$ by means of the canonical divisor $K_{\widetilde{F(V)}}$ of $\widetilde{F(V)}$.

On the one hand, note that $\widetilde{F(V)}\subset\tcR_5$ is the general fiber of ${\ocP}_5:{\ocR}_5\dasharrow\ocA_4$, the rational map extending $\tcP_5$ outside a locus of codimension $\geq2$ in $\ocR_5$ (here $\ocA_4$ denotes any toroidal compactification of $\cA_4$, see for instance \cite{cmghl}). 

It follows that $K_{\widetilde{F(V)}}=\iota^*(K_{\ocR_5})$ in $\Pic(\widetilde{F(V)})_\bQ$.
Using the expression for the canonical class $K_{\ocR_5}$ given in \cite[Theorem~1.5]{fa-lu}, we obtain
\[
K_{\widetilde{F(V)}}=\iota^*\big(13\lambda-2(\delta_0'+\delta_0'')-3\delta_0^{ram}-\ldots\big)=13\iota^*\lambda-3\iota^*\delta_0^{ram}=\iota^*\lambda.
\]

On the other hand, applying the Hurwitz formula to the \'etale double cover $\varphi$ yields an equality
\[
K_{\widetilde{F(V)}}=\varphi^*K_{F(V)}=3\varphi^*L
\]
in $\NS(\widetilde{F(V)})$.
Comparing the two expressions for $K_{\widetilde{F(V)}}$, we find the equality $\iota^*\lambda=3\varphi^*L$ in $\NS(\widetilde{F(V)})_\bQ$, and hence
\[
m\varphi^*L=\varphi^*\Gamma=8\iota^*\lambda=24\varphi^*L
\]
from which we deduce that $m=24$.
\end{proof}

Now we can answer the question above, namely, count the number of nodes in the curve $\Gamma$ for a general cubic threefold $V$.
Indeed, this number arises as the difference between the arithmetic genus and the geometric genus of $\Gamma$.

The geometric genus $g(\Gamma)$ of $\Gamma$ is that of its normalization $\Gamma'$. Since $\omega_{\Gamma'}=\cO_{\Gamma'}(\Gamma'+K_{F(V)})$ by adjunction, we have
\[
2g(\Gamma')-2=6L\cdot(6L+3L)=54L^2=270
\]
from which the equality $g(\Gamma')=136$ follows. 

On the other hand, since $\Gamma=24L$ in $\NS(F(V))$ by \autoref{classgamma}, again by adjunction the arithmetic genus of $\Gamma$ satisfies
\[
2p_a(\Gamma)-2=24L\cdot(24L+3L)=648L^2=3240
\]
and thus $p_a(\Gamma)=1621$.
Therefore, $\Gamma$ has exactly $1485$ nodes (which together with \autoref{classgamma} concludes the proof of \autoref{mainC}).

Geometrically these nodes translate as in \autoref{mainD}.\eqref{mainD:nodes}, whereas the good control of the intersection of $\Gamma$ and $\Gamma'$ reads as \autoref{mainD}.\eqref{mainD:inters}. In other words, if $V\subset\bP^4$ is a general smooth cubic threefold:
\begin{itemize}[\textbullet]
\item $V$ contains exactly $1485$ lines $l$ for which there exist $2$-planes $\pi_1,\pi_2\subset\bP^4$ and lines $r_1,r_2\subset V$ satisfying $V\cdot\pi_i=l+2r_i$.
\item $V$ contains exactly $720$ lines $l$ with the following property: there exist $2$-planes $\pi_1,\pi_2\subset\bP^4$ and lines $r_1,r_2\subset V$ such that $V\cdot\pi_1=l+2r_1$ and $V\cdot\pi_2=2l+r_2$.
\end{itemize}




Coming back to our description of the general fiber of the restricted Prym map $\tcP_5|_{\cT^o_5}$, we find:

\begin{cor}\label{DesingPrecisa}
For a general $(V,\delta)\in\cR\cC^+$, the following hold:
\begin{enumerate}[\rm (1)]
 \item $\cT^o_5\cap\widetilde{F(V)}$ is a partial desingularization of $\Gamma$, with exactly $765$ nodes.
 
 \item The intersection of $\cT^o_5\cap\widetilde{F(V)}$ and $\Delta_0^{ram}\cap\widetilde{F(V)}$ is transverse, and consists of $1440$ points.
\end{enumerate}
\begin{proof}
As usual, let us denote by $\varphi:\widetilde{F(V)}\to F(V)$ the double \'etale cover induced by $\delta$, whose associated involution exchanges the two components $\cT^o_5\cap\widetilde{F(V)}$ and $\Delta_0^{ram}\cap\widetilde{F(V)}$ of $\widetilde{\Gamma}=\varphi^{-1}(\Gamma)$.

Since the morphism $\varphi:\widetilde{\Gamma}\to\Gamma$ is \'etale and $\Gamma$ has only nodes as singularities, it follows that $\cT^o_5\cap\widetilde{F(V)}$ and $\Delta_0^{ram}\cap\widetilde{F(V)}$ have only nodes as singularities, and intersect transversely.

Conversely, the preimage of a node of $\Gamma$ must consist of:
\begin{itemize}[\textbullet]
 \item Either a node of $\cT^o_5\cap\widetilde{F(V)}$ and a node of $\Delta_0^{ram}\cap\widetilde{F(V)}$.
 \item Or two intersection points of $\cT^o_5\cap\widetilde{F(V)}$ and $\Delta_0^{ram}\cap\widetilde{F(V)}$ (where the intersection is transverse).
\end{itemize}

Therefore the proof is reduced to computing the intersection number $[\cT^o_5\cdot\Delta^0_{ram}]_{\widetilde{F(V)}}$.
Using the projection formula and \autoref{classgamma}, we deduce that
\begin{align*}
 2[\cT^o_5\cdot\Delta^0_{ram}]_{\widetilde{F(V)}}=[\cT^o_5\cdot\Delta^0_{ram}]_{\widetilde{F(V)}}+[\cT^o_5\cdot\cT^o_5]_{\widetilde{F(V)}}=[\cT^o_5\cdot\varphi^*\Gamma]_{\widetilde{F(V)}}=\\
 =[\varphi_*(\cT^o_5\cdot\varphi^*\Gamma)]_{F(V)}=[\Gamma\cdot\Gamma]_{F(V)}=2880
\end{align*}
(the first equality follows from the fact that $\cT^o_5$ and $\Delta^0_{ram}$ have the same class in $\widetilde{F(V)}$).

It turns out that $[\cT^o_5\cdot\Delta^0_{ram}]_{\widetilde{F(V)}}=1440$.
According to the previous description, these $1440$ intersection points form the preimage by $\varphi$ of $720$ nodes of $\Gamma$.
The remaining $765$ nodes of $\Gamma$ lift to nodes of $\cT^o_5\cap\widetilde{F(V)}$ and $\Delta_0^{ram}\cap\widetilde{F(V)}$.
\end{proof}
\end{cor}

\section{Genus at least 6}\label{sec:genus6imes}

\subsection{Generic finiteness on $\cT^e_g$ and $\cT^o_g$}\label{subsec:genfin}
Our first purpose in this section is to prove that, for $g\geq6$, the restrictions of the Prym map $\cP_g:\cR_g\to\cA_{g-1}$ to $\cT^e_g$ and $\cT^o_g$ are generically finite onto their image (which in particular implies \autoref{mainB}.\eqref{mainB:gg6}).
The result is actually valid for restrictions to arbitrary divisors when $g\geq8$, whereas in the cases $g=6,7$ the use of specific cohomology classes is required in our approach:

\begin{thm}\label{genfinite}
For every $g\geq6$ the restricted Prym maps $\cP_g\,|_{\cT_g^e}$ and $\cP_g\,|_{\cT_g^o}$ are generically finite onto their image.
\end{thm}
\begin{proof}
It is well known (see the proof of the main theorem and the remark in \cite{n}) that, if $(C,\eta)\in\cR_g$ is a point where the differential $d\cP_g$ fails to be injective, then $\Cliff(C)\leq2$.
According to classical results of Martens (\cite[Beispiel 7 and 8]{martens}), this means that either $C$ has a $g^1_4$ or $C$ is a plane sextic;
of course the latter may only happen if $g=10$.

If $g\geq8$, the locus in $\cM_g$ of tetragonal curves has codimension at least $2$, whereas the locus of plane sextics in $\cM_{10}$ has codimension 8. Hence, for $g\geq8$, $d\cP_g$ is injective at the general point of any divisor $D\subset\cR_g$ and $\cP_g\;|_D$ is generically finite onto its image.

If $g=7$, the Brill-Noether number $\rho(7,1,4)$ equals $-1$.
This implies (see \cite{eh}) that the Brill-Noether locus $\Mtetra$ of tetragonal curves in $\cM_7$ is an irreducible divisor, whose cohomology class in $\Pic(\ocM_7)_\bQ$ is known up to a scalar:
\[
[\Mtetra]=c\big(10\lambda -\frac 43 \delta_0
-6\delta_1-10\delta _2-12 \delta_3\big),
\]
for some $c\in \mathbb Q$.
On the other hand, the class of $\cT_7\subset \ocM_7$ is
\[
[\cT_7]= 16(129 \lambda -16 \delta_0 - 63 \delta_1-93 \delta_2 -105 \delta_3) 
\]
(see \cite{te}).
Comparing the two expressions, we conclude that the general elements of $\cT^e_7$ and $\cT^o_7$ are not covers of tetragonal curves.

For the case $g=6$, note that the locus in $\cR_6$ where the infinitesimal Prym-Torelli theorem fails is, by definition, the ramification divisor\footnote{The divisor $\cU_{6,0}$ is a particular case of the loci $\cU_{2i+6,i}\subset\cR_{2i+6}$ of Prym curves whose Prym-canonical model has a nonlinear $i$-th syzygy.
According to the \textit{Prym-Green conjecture}, every $\cU_{2i+6,i}$ is expected to be a divisor in $\cR_{2i+6}$;
see \cite[Section~8]{fa} for more details.} $\cU_{6,0}$ of the generically finite map $\cP_6$. In \cite[Theorem~0.6]{fa-lu} (see also \cite[Theorem~8.6]{fa}), using an analysis of the syzygies of Prym-canonical curves, the class of this locus is computed to be
\[
[\mathcal U_{6,0}]=7\lambda -\frac 32 \delta ^{ram}-\delta_0'-\ldots 
\]
Furthermore, the divisor $\cU_{6,0}$ is irreducible (see \cite[Theorem~0.4]{fgsmv}, where this ramification divisor is denoted by $\cQ$, and references therein).
By comparison with the cohomology classes of \autoref{class}, it follows that the supports of $\cT^e_6$ and $\cT^o_6$ are different from that of $\cU_{6,0}$, which finishes the proof.
\end{proof}

In the rest of the paper, we will focus on the case $g=6$.
Recall that the Prym map $\cP_6:\cR_6\to\cA_5$ is dominant and generically finite of degree $27$ (\cite{donsmith}).
Moreover, the correspondence induced on a general fiber by the tetragonal construction (cf. \cite{do_tetr}) is isomorphic to the incidence correspondence on the $27$ lines of a smooth cubic surface (\cite[Section~4]{do_fibres}).

Therefore, from \autoref{thetanull2} and \autoref{genfinite} one immediately deduces that $\cP_6\,|_{\cT_6^e}$ is generically finite of degree $27$, which concludes the proof of \autoref{mainA}.\eqref{mainA:gg6} and indicates that $\cT^e_6$ remains invariant under the tetragonal construction.
In fact:

\begin{prop}
Let $(C_i,\eta_i,M_i)$ ($i=1,2,3$) be a tetragonally related triple of smooth Prym curves $(C_i,\eta_i)\in\cR_6$ with a $g^1_4$ $M_i$ on $C_i$.
\begin{enumerate}[{\rm (1)}]
\item\label{item:preserveTe} If $(C_1,\eta_1)\in\cT^e_6$ is general, then $(C_2,\eta_2),(C_3,\eta_3)\in\cT^e_6$ as well.
\item\label{item:To6ToImTo7} If $(C_1,\eta_1)\in\cT^o_6$ is general, then $JC_2,JC_3\in\cP_7(\cT^o_7)$.
\end{enumerate}
\end{prop}
\begin{proof}
According to \cite[Lemma~5.5]{do_fibres}, giving the tetragonally related triple $(C_i,\eta_i,M_i)$ is equivalent to giving a trigonal curve $R\in\cM_7$ and a subgroup $W=\{0,\mu_1,\mu_2,\mu_3\}\subset JR_2$ (totally isotropic with respect to the Weil pairing), in such a way that:
\begin{itemize}[\textbullet]
 \item $(R,\mu_i)$ corresponds to $(C_i,M_i)$ under Recillas' trigonal construction.
 \item The $2$-torsion point $\eta_i\in (JC_i)_2$ is the projection of $\mu_j\in\langle\mu_i\rangle^\perp$ ($j\neq i$).
\end{itemize}

Fix a general $(C_1,\eta_1)\in \cT^e_6$ (resp., a general $(C_1,\eta_1)\in \cT^o_6$), and consider any $g^1_4$ $M_1$ on $C_1$.
By generality, $C_1$ admits a unique theta-characteristic $L_1$ with $h^0(C_1,L_1)\geq2$, which is a semicanonical pencil with $h^0(C_1,L_1)=2$ and $L_1\otimes\eta_1$ even (resp. odd).

Let $R\in\cM_7$ be the trigonal curve and let $W\subset JR_2$ be the totally isotropic subgroup defining the tetragonally related triple.
Since $\cP_7(R,\mu_1)=JC_1\in\theta_{null}\subset\cA_6$, it follows from \autoref{thetanull2}.\eqref{thetanull2:from3} that $(R,\mu_1)\in\cT^e_7$, so $R$ has a semicanonical pencil $L_R$ such that $h^0(R,L_R\otimes\mu_1)$ is even.
Moreover, if $f_1:R_1\to R$ is the cover determined by $(R,\mu_1)\in\cR_7$, then the 2-torsion singular point $L_1$ of the canonical theta divisor $\Theta_{C_1}\subset\Pic^5(C_1)$ corresponds to $f_1^*L_R$ under the identification of $\left(P^+(R,\mu_1),\Xi^+(R,\mu_1)\right)$ with $\left(\Pic^5(C_1),\Theta_{C_1}\right)$.

Now we want to determine the parity of the theta-characteristics $L_R\otimes\mu_2$ and $L_R\otimes\mu_3$.
Observe that the two parities are equal, since the Riemann-Mumford relation (see \cite{mu3} or \cite[Theorem~1.13]{harris}) gives
\[
h^0(R,L_R)+h^0(R,L_R\otimes\mu_1)+h^0(R,L_R\otimes\mu_2)+h^0(R,L_R\otimes\mu_3)\equiv\langle\mu_1,\mu_2\rangle\equiv 0\pmod{2},
\]
where we use that $\mu_3=\mu_1\otimes\mu_2$. 

Following \cite[Theorem~1.5]{do_fibres} and the notations therein, this means that for $i\in\{2,3\}$ we have $L_R\otimes\mu_i\in(\mu_1)^{\perp'}$, and hence:
\begin{align*}
h^0(R,L_R\otimes\mu_i)&\equiv
q_{JR}(L_R\otimes\mu_i)\equiv
q_{P(R,\mu_1)}(f_1^*(L_R\otimes\mu_i))\\
&\equiv
q_{JC_1}(L_1\otimes\eta_1)\equiv h^0(C_1,L_1\otimes\eta_1)\pmod{2}.
\end{align*}

If $(C_1,\eta_1)\in\cT^o_6$, then $L_1\otimes\eta_1$ is an odd theta-characteristic on $C_1$, and hence we obtain $(R,\mu_i)\in\cT^o_7$ for $i\in\{2,3\}$.
Therefore $JC_i=\cP_7(R,\mu_i)\in\cP_7(\cT^o_7)$, which proves \eqref{item:To6ToImTo7}.

If $(C_1,\eta_1)\in\cT^e_6$, then $L_1\otimes\eta_1$ is an even theta-characteristic on $C_1$ and thus $(R,\mu_2),(R,\mu_3)\in\cT^e_7$.
For $i\in\{2,3\}$ this gives $JC_i=\cP_7(R,\mu_i)\in\theta_{null}\subset\cA_6$, namely $C_i\in\cT_6$ admits a (unique, by generality) semicanonical pencil $L_i$.

Therefore, to finish the proof of \eqref{item:preserveTe} we only have to check that $(C_i,\eta_i)\in\cT^e_6$, namely that $L_i\otimes\eta_i$ is even.
This is again a consequence of \cite[Theorem~1.5]{do_fibres}:
\[
h^0(R,L_R\otimes\mu_1)\equiv q_{JR}(L_R\otimes\mu_1)\equiv q_{JC_i}(L_i\otimes\eta_i)\equiv h^0(C_i,L_i\otimes\eta_i)\pmod{2}.\qedhere
\]
\end{proof}

At present, we lack an interpretation for the Jacobian of a curve $C\in\cM_6$ being the Prym variety of a trigonal cover in $\cT^o_7$.
This prevents us from completely understanding the tetragonal construction applied to elements of $\cT^o_6$, and hence describing the (divisorial) components of $\cP_6^{-1}(\cP_6(\cT^o_6))$. 
Another natural question would be to find the degree of the map $\cP_6\,|_{\cT_6^o}$.

In this direction, partial information is obtained from cohomology classes. Once again, this reveals differences between the odd and the even cases:

\begin{prop}
$\cP_6^{-1}(\cP_6(\cT^o_6))$ contains other divisorial components apart from $\cT^o_6$.
In particular, the degree of the generically finite map $\cP_6\,|_{\cT_6^o}$ is strictly smaller than $27$.
\end{prop}
\begin{proof}
Let us denote by $\ocP:\ocR_6\dasharrow\ocA_5$ the rational Prym map obtained by extending the Prym map to the open subset of $\ocR_6$ lying over the locus of stable curves in $\ocM_6$ with at most one node.
Here $\ocA_5$ stands for the perfect cone compactification of $\cA_5$, whose rational Picard group $\Pic(\ocA_5)_\bQ$ is generated by the Hodge class $L$ and the class $D$ of the irreducible boundary divisor.

According to \cite[Theorem~7.4]{fgsmv}, the pushforwards of the basic divisor classes of $\ocR_6$ are:
\begin{align*}
 \ocP_*\lambda=18\cdot27L-57D,\;\;\ocP_*\delta_0^{ram}=4(17\cdot27L-57D),\;\;\ocP_*\delta_0'=27D,\\
 \ocP_*\delta_0''=\ocP_*\delta_i=\ocP_*\delta_{g-i}=\ocP_*\delta_{i:g-i}=0\;\text{ for $1\leq i\leq[g/2]$}
\end{align*}

On the other hand, the pullback map $\ocP^*:\Pic(\ocA_5)_\bQ\to\Pic(\ocR_6)_\bQ$ satisfies
\[
\ocP^*L=\lambda-\frac{1}{4}\delta_0^{ram},\qquad \ocP^*D=\delta_0'
\]
(see \cite[Theorem~5]{gsm}).
The boundary divisors $\delta_0'',\delta_i,\delta_{g-i},\delta_{i:g-i}$ do not appear since they are contracted by $\ocP$.

Using the cohomology class $[\cT^o_6]\in\Pic(\ocR_6)_\bQ$ given by \autoref{class}, we have
\[
\ocP_*[\cT^o_6]=10584L-1320D
\]
Observe that this class equals $d\cdot[\cP_6(\cT^o_6)]$, where $d=\deg(\cP_6\,|_{\cT_6^o})$ and $[\cP_6^{}(\cT^o_6)]\in\Pic(\ocA_5)_\bQ$ is the class of (the closure in $\ocA_5$ of) $\cP_6(\cT^o_6)$.
Pulling back we obtain
\[
\ocP^*\ocP_*[\cT^o_6]=10584\lambda-1320\delta_0'-2646\delta_0^{ram}
\]
Since these coefficients are not proportional to the corresponding ones in $[\cT^o_6]$, it follows that $\cT^o_6$ cannot be the unique divisorial component of $\cP_6^{-1}(\cP_6^{}(\cT^o_6))$.
\end{proof}

\subsection{$\cT^o_6$ and singular surfaces of twice the minimal class}\label{subsec:singV2} 
In this final subsection we give the first steps towards an intrinsic description of the locus $\cP_6(\cT^o_6)$ in $\cA_5$, with the help of Brill-Noether loci on Prym varieties.
In order to be consistent with the notation in the proof of \autoref{genfinite}, we denote by $\cU_{6,0}\subset\cR_6$ the ramification divisor of $\cP_6$.

Recall that the Andreotti-Mayer locus $\cN_0$ in $\cA_5$ is the union of two irreducible divisors $\theta_{null}$ and $\cN_0'$.
The theta divisor of a general element of $\theta_{null}$ has a unique singular point (which is $2$-torsion), whereas the theta divisor of a general element of $\cN_0'$ has exactly two singular (opposite) points.

The relation between $\cP_6$ and the component $\cN_0'$ of the Andreotti-Mayer locus in $\cA_5$ is described in \cite[Sections~6 and 7]{fgsmv}.
In particular, the following statements hold:

\begin{enumerate}[(1)]
 \item The divisor $\cN_0'\subset\cA_5$ is the branch divisor of $\cP_6$ (see also \cite[pages 93 and 97]{do_fibres}).
 \item The preimage $\cP_6^{-1}(\cN_0')$ has two divisorial components: the ramification divisor $\cU_{6,0}$ and an \emph{antiramification divisor} $\cU$. As cycles, there is an equality
 \[
 \cP_6^*\cN_0'=2\cU_{6,0}+\cU.
 \]
 \item $\cU_{6,0}$ is the set of $(C,\eta)\in\cR_6$ for which $V^3(C,\eta)\neq\emptyset$ (i.e., the theta divisor of $P(C,\eta)$ has a stable singularity), and is mapped six-to-one to $\cN_0'$ (see \cite[Corollary~2.3]{do_tetr}).
 \item $\cU=\pi^*(\GPetri)$ is the pullback to $\cR_6$ of the Gieseker-Petri locus
 \[
 \GPetri=\set{C\in\cM_6\mid \exists L\in W^1_4(C)\text{ such that the Petri map $\mu_{0,L}$ is not injective}}
 \]
 and is mapped fifteen-to-one to $\cN_0'$.
\end{enumerate}

As usual, for $(C,\eta)\in\cR_6$ let us denote by $f:\widetilde{C}\to C$ the corresponding double \'etale cover, and by $\sigma:\widetilde{C}\to\widetilde{C}$ the involution exchanging sheets.

\begin{prop}\label{singV2}
If $(C,\eta)\in\cR_6$ is a non-hyperelliptic Prym curve with $V^4(C,\eta)=\emptyset$, then $V^2(C,\eta)$ is singular if and only if $(C,\eta)\in\cU\cup\cT^o_6$.
\end{prop}
\begin{proof}
According to \autoref{V2}, singular points $M\in V^2(C,\eta)$ are exactly those of the form
\[
M=f^*L\otimes A
\]
with $h^0(C,L)\geq2$ and $h^0(\tC,A)>0$.
In order to prove the statement, we distinguish the possible values of $d=\deg L$ allowing the existence of such an $M$.

For $d\leq4$, this condition is equivalent to the existence of $L\in W^1_d(C)$ and $A$ effective satisfying $h^0(\tC,f^*L\otimes A)=3$ and $\omega_C=L^2\otimes\Nm_f(A)$.
This happens if and only if there exists $L\in W^1_d(C)$ with $\omega_C\otimes L^{-2}$  effective: the ``only if'' part being clear, if $\omega_C\otimes L^{-2}$ is effective then Mumford's parity trick (\cite[bottom of page~186]{mu2}) allows us to find $A$ effective with $\Nm_f(A)=\omega_C\otimes L^{-2}$ and $h^0(\tC,f^*L\otimes A)=3$.

If there exists $L\in W^1_3(C)$, then one immediately checks that $\omega_C\otimes L^{-2}$ is effective.
Moreover, take $x\in C$ such that $\omega_C\otimes L^{-2}(-x)$ is effective. Since $\Cliff(C)\geq1$ by assumption, one has $h^0(C,L(x))=2$ and the kernel of the Petri map
\[
\mu_{0,L(x)}:H^0(C,L(x))\otimes H^0(C,\omega_C\otimes L^{-1}(-x))\longrightarrow H^0(C,\omega_C)
\]
is $\ker(\mu_{0,L(x)})\cong H^0(C,\omega_C\otimes L^{-2}(-x))\neq0$ by the base-point-free pencil trick (\cite[page~126]{acgh}).
In other words, the two statements $\Sing V^2(C,\eta)\neq\emptyset$ and $(C,\eta)\in\cU$ hold whenever $C$ is trigonal.

Now assume that $C$ is not trigonal.
We claim that the existence of $L\in W^1_4(C)$ with $\omega_C\otimes L^{-2}$ effective is equivalent to $C\in\GPetri$, namely to $(C,\eta)\in\cU$.
Indeed, if one can write $\omega_C=L^2(a+b)$ for points $x,y\in C$, then $\omega_C\otimes L^{-1}=L(x+y)$ and thus the Petri map $\mu_{0,L}$
fails to be injective.
Conversely, if the Petri map
 \[
 \mu_{0,L}:H^0(C,L)\otimes H^0(C,\omega_C\otimes L^{-1})\longrightarrow H^0(C,\omega_C)
 \]
of a certain $L\in W^1_4(C)$ has nonzero kernel, then the line bundle $\omega_C\otimes L^{-2}$ is effective since $\ker(\mu_{0,L})\cong H^0(C,\omega_C\otimes L^{-2})$ by the base-point-free pencil trick.

Now we can assume that there exists no $L\in W^1_4(C)$ with $\omega_C\otimes L^{-2}$ effective.
It only remains to check the case $d=5$: the condition reads $f^*L\in V^2(C,\eta)$, which is equivalent to the existence of a theta-characteristic $L$ on $C$ with
\[
h^0(C,L)+h^0(C,L\otimes\eta)=3.
\]
Under our assumption on $W^1_4(C)$, this is equivalent to either $L$ or $L\otimes\eta$ being an odd semicanonical pencil for the cover $f$, namely to $(C,\eta)\in\cT^o_6$.
\end{proof}

\begin{rem}\label{twosing}
Let us comment first on the hypothesis $V^4(C,\eta)=\emptyset$. 
\begin{enumerate}[{\rm (1)}]
 \item If $V^4(C,\eta)\neq\emptyset$, then $V^2(C,\eta)$ is automatically singular at points $M\in V^4(C,\eta)$ (this is an immediate application of \cite[Proposition~IV.4.2]{acgh}).
 
 \noindent On the other hand, for $M\in V^4(C,\eta)$ one immediately deduces from Mumford's parity trick that $M(x-\sigma(x))\in V^3(C,\eta)$ for every $x\in\widetilde{C}$. 
 As a consequence, $V^3(C,\eta)$ is at least $1$-dimensional whenever $V^4(C,\eta)\neq\emptyset$ (in particular, $(C,\eta)\in\cU_{6,0}$).
 
 \item Let $C\in\GPetri$ be general, so that there is a unique $L\in W^1_4(C)$ and unique $x,y\in C$ satisfying $L^2(x+y)=\omega_C$.
 If $\widetilde{x},\sigma(\widetilde{x})$ (resp. $\widetilde{y},\sigma(\widetilde{y})$) are the two points of $\tC$ lying over $x$ (resp. over $y$), then the four candidates for a singularity of $V^2(C,\eta)$ are:
 \[
 f^*L(\widetilde{a}+\widetilde{b}),\;\;f^*L(\sigma(\widetilde{a})+\widetilde{b}),\;\;f^*L(\widetilde{a}+\sigma(\widetilde{b})),\;\;f^*L(\sigma(\widetilde{a})+\sigma(\widetilde{b})).
 \]

 \noindent By Mumford's parity trick, these line bundles can be divided into two pairs (namely $f^*L(\widetilde{a}+\widetilde{b}),f^*L(\sigma(\widetilde{a})+\sigma(\widetilde{b}))$ and $f^*L(\sigma(\widetilde{a})+\widetilde{b}),f^*L(\widetilde{a}+\sigma(\widetilde{b}))$) according to the component of $\Nm^{-1}(\omega_C)=P^+\cup P^-$ in which they live.
 
 \noindent Since $(C,\eta)\notin\cU_{6,0}$ by genericity (and thus $V^3(C,\eta)=\emptyset$), it follows that two of them satisfy $h^0=2$ and the other two satisfy $h^0=3$.
 In other words, for a general $(C,\eta)\in\cU$ the theta divisor $\Xi^+$ has two exceptional singularities (hence, $P(C,\eta)\in\cN_0'$) and the Brill-Noether locus $V^2(C,\eta)$ has two singular points.
\end{enumerate}
\end{rem}

Let us recall that, for a non-hyperelliptic Prym curve $(C,\eta)=(\tC,C)\in\cR_g$, the Brill-Noether locus $V^2(C,\eta)$ (when properly translated to $P(C,\eta)$) is a subvariety of twice the minimal class (see \cite[Theorem~9]{decprag} and \cite[Corollary~4.4]{ln}).
Moreover $V^2(C,\eta)$ is symmetric in $P(C,\eta)$, if the translation is performed with a theta-characteristic of $\tC$ lying in $P^-$.

Combining this observation with \autoref{singV2} and \autoref{twosing}, it is tempting to propose the following analogue of the decomposition of the Andreotti-Mayer locus:

\begin{quest}\label{conjA5}
Let $\cV\subset\cA_5$ be (the closure of) the locus of ppav $(A,\Theta)$ containing an integral surface $S$ with the following properties:
\begin{enumerate}[{\rm (1)}]
 \item The cohomology class of $S$ is twice the minimal class: $[S]=2\frac{[\Theta]^3}{6}$ in $H^6(A,\bZ)$.
 \item $S$ is symmetric.
 \item $S$ has singular points.
\end{enumerate}
Does $\cV$ decompose as the union of two irreducible divisors?
Namely, the closure of $\cP_6(\cT_6^o)$ (whose general element contains at least one such surface $S$ with a unique singular point, which is 2-torsion) and $\cN_0'$ (whose general element contains fifteen such surfaces $S$, with two singular opposite points each).
\end{quest}

\vspace{1.5mm}

\bibliography{ref}
\bibliographystyle{alphaspecial}

\end{document}